\documentclass[11pt]{amsart}  
\usepackage{amsthm, amsmath, amscd, amssymb, latexsym, stmaryrd, color}

\usepackage[all]{xypic}
\usepackage{mathtools}
\input xypic

\usepackage[top=3.0cm, bottom=3.0cm, left=3.0cm, right=3.0cm]{geometry}

\theoremstyle{plain}

\newtheorem{theorem}{Theorem}[section]
\newtheorem{theoreme}{Theorem}
\newtheorem{conjj}{Conjecture}
\newtheorem{lemma}[theorem]{Lemma}
\newtheorem{lemma-def}[theorem]{Lemma-Definition}
\newtheorem{proposition}[theorem]{Proposition}
\newtheorem{prop-def}[theorem]{Proposition-Definition}
\newtheorem{corollary}[theorem]{Corollary}
\newtheorem{conjecture}[theorem]{Conjecture}

\theoremstyle{definition}

\newtheorem{definition}[theorem]{Definition}

\newtheorem{remark}[theorem]{Remark}
\newtheorem{example}[theorem]{Example}

\theoremstyle{remark}
\newtheorem*{ack}{Acknowledgement}

\usepackage[mathscr]{eucal}
\usepackage{graphics, graphpap}
\usepackage{array, tabularx, longtable}
\usepackage{color}
\numberwithin{equation}{section}

\def\Var{\mathrm{Var}}

\def\loc{\mathrm{loc}}
\def\sp{\mathrm{sp}}

\def\sing{\mathrm{sing}}

\def\ord{\mathrm{ord}}

\def\Spec{\mathrm{Spec}}

\def\Spf{\mathrm{Spf}}
\def\Sp{\mathrm{Sp}}

\def\rig{\mathrm{rig}}

\def\Gr{\mathrm{Gr}}

\def\sr{\mathrm{sr}}
\def\Hom{\mathrm{Hom}}

\def\MV{\mathrm{MV}}

\def\flat{\mathrm{flat}}

\def\x{\mathbf{x}}

\def\L{\mathbb{L}}
\def\R{\mathbb{R}}
\def\Q{\mathbb{Q}}

\title[Motivic integration on rigid varieties and integral identity conjecture]{Motivic integration on special rigid varieties and the motivic integral identity conjecture}  

\author{Nguyen Hong Duc}
\address{$^{\dag}$TIMAS, Thang Long University, \newline \indent Nghiem Xuan Yem, Hanoi, Vietnam.} 
\email{duc.nh@thanglong.edu.vn}
\thanks{}


\begin{document}           

\begin{abstract}

We prove in this paper the original version of  Kontsevich and Soibelman's motivic integral identity conjecture for formal functions by developing a novel framework for equivariant motivic integration on special rigid varieties. This theory is built upon our recent research on equivariant motivic integration within the realm of special formal schemes. The central element of our approach lies in demonstrating that two formal models of a given smooth rigid variety can be dominated by a third formal model. Notably, a similar assertion for quasi-compact rigid varieties was obtained by Bosch, L\"utkebohmert, and Raynaud in 1993. Consequently, we establish a concept of motivic volume for a special smooth rigid variety, ensuring independence from the selection of its models. We demonstrate that this motivic volume can be extended to a homomorphism from a certain Grothendieck ring of special smooth rigid varieties to the classical Grothendieck ring of varieties. Moreover, our developed motivic volume exhibits a Fubini-type property, which recovers Nicaise and Payne's motivic Fubini theorem for the tropicalization map.

\end{abstract}

\maketitle  
\tableofcontents
\addtocontents{toc}{\protect\setcounter{tocdepth}{1}}
\section{Introduction}
We construct in this paper an equivariant motivic integration on rigid varieites after quickly revisiting the  theory of equivariant motivic integration on on special formal schemes. As an application, we give a proof of Kontsevich and Soibelman's  motivic integral identity conjecture.
\subsection{}
Introduced by Kontsevich in 1995, motivic integration has since become an important subject in algebraic geometry by virtue of its connection to many areas of mathematics, including mathematical physics, birational geometry, non-Archimedean geometry, tropical geometry, singularity theory, Hodge theory, model theory (see for instance, \cite{DL1}, \cite{DL2}, \cite{Ba}, \cite{Se}, \cite{LS}, \cite{HK}, \cite{NS}, \cite{Ni2}, \cite{CL08}, \cite{KS}). Kontsevich's method involves arc spaces and the Grothendieck ring of varieties, which brings about the birth of geometric motivic integration. Nowadays, this kind of integration becomes one of the common central objects of algebraic geometry, singularity theory, mathematical physics. From algebraic varieties to formal schemes, the development of geometric motivic integration is contributed crucially by Denef-Loeser \cite{DL1, DL2}, Sebag \cite{Se}, Loeser-Sebag \cite{LS}, Nicaise-Sebag \cite{NS1, NS, NS2, NS3}, Nicaise \cite{Ni2}, and many others. Another point of view on motivic integration known as arithmetic motivic integration was also strongly developed due to the approaches of Denef-Loeser over $p$-adic fields \cite{DL3}, Cluckers-Loeser \cite{CL05, CL08, CL10}, Hrushovski-Kazhdan \cite{HK} and Hrushovski-Loeser \cite{HL16} using model theory with respect to different languages. It is shown in \cite{CHL} that arithmetic motivic integration has an important application to the fundamental lemma. For another theory of motivic integration that specializes to both of arithmetic and geometric points of view, we can also refer to more recent works such as \cite{GY}, \cite{CGH}, \cite{CR}. 
\subsection{}\label{sec1.2}
Let $R$ be a complete discrete valuation ring with fraction field $K$ and perfect residue field $k$. Let $\varpi\in R$ be a uniformizing parameter, which will be fixed throughout this article. We denote by $R^{sh}$ and $K^{sh}$ the strict henselizations of $R$ and $K$ respectively. A special formal $R$-scheme is a separated Noetherian adic formal scheme locally defined as $\Spf A$ where $A$ is a Noetherian $R$-algebra with the largest  ideal of definition $J$ such that $A/J$ is a finitely generated $k$-algebra. We associate to each special formal $R$-scheme its reduction and generic fiber. Let $\mathfrak X$ be a special formal $R$-scheme with the largest ideal of definition $J$. The reduction $\mathfrak X_0$ of $\mathfrak X$ is the closed subvariety over $k$ defined by $J$. The generic fiber $\mathfrak X_\eta$ of $\mathfrak X$ was constructed by Berthelot (\cite[0.2.6]{Berthelot}), it is a rigid variety over $K$ (see Definition \ref{def21}). Let $\omega$ be a gauge form on the generic fiber $\mathfrak X_\eta$ of $\mathfrak X$. Let $G$ be a finite algebraic group over $k$ which acts on $\mathfrak X$. Based on the works in \cite{Se,LS,Ni2} we define in \cite{LN22} a notion of motivic $G$-integral of $\omega$ on $\mathfrak X$, which generalizes the integral in \cite{Ni2} to equivariant setting. This integral is denoted by $\int_{\mathfrak X}|\omega|$ and  takes its value in the ring $\mathscr M_{\mathfrak X_0}^{G}$, a localization of the Grothendieck ring of varieties. It admits a change of variable formula (Proposition \ref{changeofvariables}). Moreover, it is additive with respect to open covers and with respect to the completions along strata of locally closed stratifications of the reduction $\mathfrak X_0$ of $\mathfrak X$ (Proposition \ref{int-additive}). 

Let $\hat\mu$ be the profinite group scheme of roots of unity, i.e.~ the projective limit of the group schemes $\mu_n=\Spec \left(k[\xi]/(\xi^n-1)\right)$. For  $n\in \mathbb N^*$ and  any formal $R$-scheme $\mathfrak X$, we put $R(n)=R[\tau]/(\tau^n-\varpi)$, $K(n)=K[\tau]/(\tau^n-\varpi)$ and $\mathfrak X(n):=\mathfrak X\times_RR(n)$. Then for any gauge form $\omega$ on $\mathfrak X_\eta$ we consider the $\mu_n$-integral of $\omega$ on $\mathfrak X(n)$, where the action $\mu_n$ on $\mathfrak X(n)$ is induced from the natural action of $\mu_n$ on $R(n)$.  We define the Poincar\'e series of $\mathfrak X,\omega$ as 
$$P(\mathfrak X,\omega; T):=\sum_{n\geq 1}\Big(\int_{\mathfrak X(n)}|\omega(n)|\Big)T^n\ \in \mathscr M_{\mathfrak X_0}^{\hat\mu}[\![T]\!].$$
and prove that this series is rational if the characteristic of $k$ is zero. Let $d$ be the dimension of $\mathfrak X$ over $R$. We define the motivic volume of $\mathfrak X$ as
$$\MV(\mathfrak X):=-\L^d\lim_{T\to\infty}P(\mathfrak X,\omega; T),$$
which is independent of the choice of $\omega$.

Assume that the characteristic of $k$ is zero. Let $f\in k\{x\}[\![y]\!]$, with $x=(x_1,\dots, x_m)$ and $y=(y_1,\dots,y_{m'})$ such that $f(x,0)$ is not a non-zero constant. Let $\mathfrak X(f)$ denote the formal completion of $\Spf(k\{x\}[\![y]\!])$ along $(f)$ with the $R$-structural morphism induced by $\varpi \mapsto f$. The motivic volume of $\mathfrak X(f)$  is denoted by $\mathscr S_f$ and called the the {\em motivic nearby cycle of $f$}. If $f$ is a polynomial, this recovers the notion of motivic nearby cycles defined by Denef-Loeser (\cite{DL1}, \cite{DL2}).

\subsection{}\label{sec1.3}

A rigid $K$-variety $X$ is called {\em special} if it admits a formal model which is a special formal $R$-scheme. That is, there exists a special formal $R$-scheme $\mathfrak X$ such that $\mathfrak X_\eta=X$. An action of $G$ on $X$ is an equivalent class of the pair $(\mathfrak X,\theta)$ consisting of a model $\mathfrak X$ endowed with an action $\theta$ of $G$. Two $G$-pairs $(\mathfrak X,\theta)$ and $(\mathfrak X',\theta')$ of $X$ are {\it equivalent} if there exist a formal $R$-scheme $\mathfrak X''$ endowed with a good $G$-action (see Section \ref{section23}) and two $G$-equivariant morphisms $\mathfrak X''\to \mathfrak X'$ and $\mathfrak X''\to \mathfrak X$ such that the induced morphisms morphism $\mathfrak X''_\eta\to \mathfrak X'_\eta$ and $\mathfrak X''_\eta\to \mathfrak X_\eta$ are open embedding satisfying $\mathfrak X''_{\eta}(K^{sh})=X(K^{sh})$.  If $\omega$ is a gauge form on $X$, then the quantity
$$\int_{\mathfrak X_0}\int_{\mathfrak X}|\omega|\ \in \mathscr M_{k}^G$$
is independent of the choice of a representative $(\mathfrak X,\theta)$ and called the {\em $G$-integral} of $\omega$ on $X$ and denoted by $\int_{X}|\omega|$. Considering the natural action of $\mu_n$ on $X(n):=X\otimes K(n)$ we obtain Poincar\'e series of $X,\omega$ defined as
$$P(X,\omega; T):=\sum_{n\geq 1}\Big(\int_{X(n)}|\omega(n)|\Big)T^n\ \in \mathscr M_{k}^{\hat\mu}[\![T]\!].$$
We prove that if $X$ is bounded (see Definition \ref{boundedvar}), then the series $P(X,\omega; T)$ is rational.

We would like to define the motivic volume of a smooth special rigid $K$-variety $X$. If $X$ is quasi-compact, this definition was given in \cite[8.3]{NS}: the image of $\MV\left(\mathfrak{X}\right)$ under the forgetful morphism $\mathcal{M}_{\mathfrak{X}_{0}} \rightarrow \mathcal{M}_{k}$ only depends on $\mathfrak{X}_{\eta}$. In order to define the motivic volume for  smooth special rigid $K$-varieties, we need to prove that ``any two formal models of $X$ can be dominated by a third", which was asksed by J. Nicaise in \cite[Page 338]{Ni2}. The same claim for quasi-compact varieties was proved in \cite{BLR93}, but the method thereby can not be extended to the special rigid varieties. For a proof of the statement, we need to apply the de Jong's descent theory for closed rigid subvarieties in \cite{deJ}.
\begin{theoreme}[Theorem \ref{modelization}]\label{modelization1}
Let $X$ be a smooth special rigid $K$-variety. If $\mathfrak X$ and $\mathfrak X'$ are two formal models of $X$, then there exist another model $\mathfrak X''$ of $X$ together with two morphisms $h: \mathfrak X''\to \mathfrak X$ and $h': \mathfrak X''\to \mathfrak X'$ such that the induced morphisms ${h}_\eta$ and ${h'}_\eta$ are isomorphisms. 
\end{theoreme}


Applying Theorem \ref{modelization1} one can show that the quantity
$$\MV(X):=\int_{\mathfrak X_0}\MV(\mathfrak X)$$
depends only on $X$, and called the {\em motivic volume } of $X$. 

Let $\mathrm{SSRig}_K$ denote the category of special smooth rigid $K$-varieties. A special rational subdomain of $X$ is defined locally as 
\begin{align*}
X\Big(\frac{f}{g}\Big):=\left\{x\in X \mid |f_i(x)|\leq |g(x)|,\ \forall i\right\},
\end{align*}
where $X=(\Spf A)_\eta$ and $g,f_1,\ldots,f_n\in A$ generating the unit ideal in $ A{\otimes}_R K$. Define $\mathbf{K}(\mathrm{SSRig}_K)$ the abelian group generated by the isomorphism classes $[X]$ of $\mathrm{SSRig}_K$ modulo the relation 
$$[X]=[Y]+[X\setminus Y]$$
where $Y\subseteq X$ is a special rational subdomain of $X$. The group $\mathbf{K}(\mathrm{SSRig}_K)$ admits a ring structure whose multiplication is induced by fiber product.
\begin{theoreme}[Theorem \ref{MV-morphism}]\label{MV-morphism}
There exists a unique ring homomorphism
$$\MV\colon \mathbf{K}(\mathrm{SSRig}_K)\to \mathscr M_k^{\hat{\mu}}$$
satisfying
$$\MV([X])=\MV(X)$$ 
for all objects $X$ of $\mathrm{SSRig}_K$.
\end{theoreme}
The motivic volume $\MV$ admits the following version of Fubini theorem.
\begin{theoreme}[Theorem \ref{fubini}]
Let $X$ be a smooth special rigid $K$-varitety with a model $\mathfrak X$. Let $g=\{g_1,\ldots,g_r\}$ be a system of elements of $\Gamma(\mathfrak X,\mathcal{O}_{\mathfrak X})$. For each $\gamma\in  \Q^r_{\geq 0}$ we define the variety
$$X_{\gamma}:=\left \{x\in X\mid |g_i(x)| = |\varpi|^{\gamma_i} \right\}.$$
Then the function $\varphi_g\colon \Bbb Q^r_{\geq 0}\to \mathscr M_k^{\hat{\mu}}$ defined as
$$\varphi_g(\gamma)=\MV(X_{\gamma})$$
is constructible, and moreove, 
$$\MV(X)=\int_{ \Bbb Q^r_{\geq 0}} \varphi_g d\chi_c=\int_{ \Bbb Q^r_{\geq 0}}\varphi_g d \chi^{\prime}.$$
Here $\chi_c$ and $ \chi^{\prime}$ denote the compactly supported and bounded Euler characteristics respectively.
\end{theoreme}

\subsection{}
We explore a vital use of $\hat\mu$-equivariant motivic integration on rigid $K$-varieties in relation to the integral identity conjecture. It is widely recognized that this conjecture serves as a foundational element in Kontsevich-Soibelman's theory of motivic Donaldson-Thomas invariants concerning noncommutative Calabi-Yau threefolds. Specifically, it directly implies the existence of these invariants, as outlined in \cite{KS}. Let us first state the conjecture, see \cite[Conjecture 4.4]{KS}. Here, for a tuple of variables $x=(x_1,\ldots,x_d)$ and a new univariate $t$ we write $tx$ for $(tx_1,\ldots,tx_d)$.

\begin{conjj}[Kontsevich-Soibelman]\label{conj}
Let $f\in k[\![x,y,z]\!]$  be a formal power series such that $f(0,0,0)=0$ and  $f(tx,y,z)=f(x,ty,z)$ in $k[\![x,y,z,t]\!]$, where $x,y,z$ are tuples of variables $x=(x_1,\ldots,x_{d_1}), y=(y_1,\ldots,y_{d_2})$ and $z=(z_1,\ldots,z_{d_3})$. Then $f$ is considered as an element in $k\{x\}[\![y,z]\!]$ and then the identity
\begin{align*}
\int_{\Bbb A^{d_1}_k} \mathscr S_f=\L^{d_1}  \mathscr S_{\tilde f,0}
\end{align*}
holds in $ \mathscr M_k^{\hat{\mu}}$, where $\tilde f(z)=f(0,0,z)\in k[\![z]\!]$.
\end{conjj}
As far as we know there has been no proofs of the conjecture for the case of formal series even the special case when $k$ is assumed to be algebraically closed. All known results are assumed $f$ to be a polynomial. More precisely, the conjecture for polynomials was first proved by L\^e \cite{Thuong1} for the case where $f$ is either a function of Steenbrink type or the composition of a pair of regular functions with a polynomial in two variables. In \cite[Theorem 1.2]{Thuong}, in view of the formalism of Hrushovski-Kazhdan \cite{HK} and Hrushovski-Loeser \cite{HL16}, L\^e showed that the conjecture for polynomials holds in $\mathscr M_{\loc}^{\hat\mu}$, a ``big'' localization of $\mathscr M_k^{\hat\mu}$, as soon as the base field $k$ is algebraically closed. Nicaise and Payne \cite{NP} proved the conjecture for polynomials by proving the motivic Fubini theorem for the tropicalization map, on the foundation of \cite{HK} and tropical geometry, with assumption that $k$ contains all roots of unity. In \cite{LN}, by developing an equivariant motitic integration on varieties, we give a proof of the conjecture for polynomials.

In this article, we give a proof of Conjecture \ref{conj} (see Theorem \ref{iic}). The work is based on our theory of equivariant motivic integration on special rigid varieties developed in Section \ref{rig} (see Section \ref{proofofiic} for detailed arguments). 
\addtocontents{toc}{\protect\setcounter{tocdepth}{2}}
\section{Equivariant motivic integration on special formal schemes}
In this section we recall the theory of equivariant motivic integration for special formal schemes developed in \cite{LN22}.  All definitions and results are borrowed from \cite{LN22} and \cite{Ni2} except Proposition \ref{contactinvariant}.
\subsection{Equivariant Grothendieck rings of varieties} \label{Sec2.1}

Let $k$ be a perfect field. Let $S$ be a $k$-variety endowed with a good action of a finite algebraic group $G$. We denote by $\Var_S^G$ the category of $S$-varieties $X$ endowed with a good action of $G$ such that the morphism $X\to S$ is $G$-equivariant. By definition, $\mathbf{K}(\Var_S^G)$ is the quotient of the free abelian group generated by the $G$-equivariant isomorphism classes $[X]$ in $\Var_S^{G}$ modulo the relations
$$[X\to S]=[Y\to S]+[X\setminus Y\to S],$$
for $Y$ being $G$-invariant and Zariski closed in $X$, and 
\begin{align}\label{equiv}
[X\times_k\mathbb A_k^n\to S,\sigma]=[X\times_k\mathbb A_k^n\to S,\sigma']
\end{align}
if $\sigma$ and $\sigma'$ lift the same good $G$-action on $X$. Together with fiber product over $S$, $\mathbf{K}(\Var_S^G)$ is a commutative ring with unity $\mathrm{id}_S$. We define the localization $\mathscr M_S^G$ of the ring $\mathbf{K}(\Var_S^G)$ by inverting $\L$ where $\L$ is the class of $\mathbb A_k^1\times_k S\to S$ endowed with the trivial action of $G$. 

Let $\hat G$ be a group scheme over $k$ of the form $\hat G=\varprojlim_{i\in I} G_i$, where $I$ is a partially ordered set and $\left\{G_i, G_j\to G_i \mid i\leq j \ \text{in}\ I\right\}$ is a projective system of algebraic groups over $k$. We define $K_0^{\hat G}(\Var_S)=\varinjlim_{i\in I}  K_0^{G_i}(\Var_S)$ and $\mathscr M_S^{\hat G}=K_0^{\hat G}(\Var_S)[\L^{-1}]$, which implies the identity $\mathscr M_S^{\hat G}=\varinjlim_{i\in I}  \mathscr M_S^{G_i}$.  In particular, we may consider $\hat G$ to be the profinite group scheme of roots of unity $\hat\mu$, the projective limit of the group schemes $\mu_n=\Spec \left(k[\xi]/(\xi^n-1)\right)$ and transition morphisms $\mu_{mn}\to \mu_n$ sending $\lambda$ to $\lambda^m$.

Let $f\colon S\to S'$ be a morphism of algebraic $k$-variety. We denote by $f^*\colon  \mathscr M_{S'}^G\to \mathscr M_{S}^G$ the ring homomorphism induced from the fiber product (the pullback morphism), and by $f_!\colon \mathscr M_{S}^G\to  \mathscr M_{S'}^G$ the $ \mathscr M_{k}^G$-linear homomorphism defined by the composition with $f$ (the push-forward morphism). When $S'$ is $\Spec k$, one usually writes $\int_S$ instead of $f_!$.


\subsection{Rational series}
Let $\mathscr M$ be a commutative ring with unity which contains $\L$ and $\L^{-1}$. Let $\mathscr M[\![T]\!]$ be the set of formal power series in $T$ with coefficients in $\mathscr M$, which is a ring and also a $\mathscr M$-module with respect to usual operations for series. Denote by $\mathscr M[\![T]\!]_{\sr}$ the submodule of $\mathscr M[\![T]\!]$ generated by 1 and by finite products of terms $\frac{\L^aT^b}{(1-\L^aT^b)}$ for $(a,b)\in\mathbb{Z}\times\mathbb{N}_{>0}$. An element of $\mathscr M[\![T]\!]_{\sr}$ is called a {\it rational series}. By \cite{DL1}, there exists a unique $\mathscr M$-linear morphism 
$$\lim_{T\to\infty}: \mathscr M[\![T]\!]_{\sr}\to \mathscr M$$ 
such that for any $(a,b)$ in $\mathbb{Z}\times\mathbb{N}_{>0}$, one has
$$\lim_{T\to\infty}\frac{\L^aT^b}{(1-\L^aT^b)}=-1.$$ 
Let $I$ be a finite set.  Let $\Delta$ be a rational polyhedral convex cones in $\mathbb{R}_{>0}^{I}$. This means that, $\Delta$ is a convex subset of $\mathbb{R}_{>0}^{l}$ defined by a finite number of integral linear inequalities of type $a \geq 0$ or $b>0$ and stable by multiplication by $\mathbb{R}_{>0}$. We denote by $\bar{\Delta}$ the closure of $\Delta$ in $\mathbb{R}_{\geq 0}^{I}$. Let $\ell$ and $v$ be integral linear forms on $\mathbb{Z}^{I}$ which are positive on $\bar{\Delta} \backslash\{0\}$. It follows from \cite[2.9]{GLM} that
\begin{align}\label{eq3.1}
\lim_{T\to\infty} \sum_{k \in \Delta \cap \mathbb{N}_{>0}^{I}} T^{\ell(k)} \mathbb{L}^{-v(k)}=\chi_c(\Delta)
\end{align} 
in $\mathscr M[\![T]\!]$.

We will use the notion of the Hadamard product of two formal power series. By definition, the Hadamard product of two formal power series $p(T)=\sum_{n\geq 1}p_nT^n$ and $q(T)=\sum_{n\geq 1}q_nT^n$ in $\mathscr M[\![T]\!]$ is the series
\begin{align*}
p(T)\ast q(T):=\sum_{n \geq 1}p_n\cdot q_nT^n\in \mathscr M[\![T]\!].
\end{align*} 
\begin{lemma}[\cite{Loo}]\label{Lem2}
If $p(T)$ and $q(T)$ are rational series in $\mathscr M[\![T]\!]$, so is $p(T)\ast q(T)$, and in this case,
$$\lim_{T\to\infty}p(T)\ast q(T)=-\lim_{T\to\infty}p(T) \cdot \lim_{T\to\infty}q(T).$$
\end{lemma}


\subsection{Special formal schemes with actions}\label{section23}
Let $R$ be a complete discrete valuation ring with fraction field $K$ and residue field $k$. Let $\varpi\in R$ be a uniformizer, which will be fixed throughout this article. We denote by $R^{sh}$, $K^{sh}$ the strict henselization of $R$ and $K$ respectively. A topological $R$-algebra $A$ is called {\it special} if $A$ is a Noetherian adic ring and the $R$-algebra $A/J$ is finitely generated for some ideal of definition $J$ of $A$. Let $R\{x_1,\dots, x_m\}[\![y_1,\dots,y_{m'}]\!]$ be the subring of elements of the form
$$\sum_{\alpha}c_{\alpha}x^{\alpha},\ c_{\alpha}\in R[\![y_1,\dots,y_{m'}]\!]$$ such that $|c_{\alpha}|\to 0$ as $|\alpha|\to \infty$, where the norm in  $R[\![y_1,\dots,y_{m'}]\!]$ is induced from the order.  By \cite{Ber96}, a topological $R$-algebra $A$ is special if and only if $A$ is topologically $R$-isomorphic to a quotient the $R$-algebra $R\{x_1,\dots, x_m\}[\![y_1,\dots,y_{m'}]\!]$ for some $m, m'\in\mathbb N^*$.

\begin{definition}\label{def21}
A {\it special} formal scheme is a separated Noetherian adic formal scheme $\mathfrak X$ which is a finite union of open affine formal schemes of the form $\Spf A$ with $A$ a Noetherian special $R$-algebra. If $\mathfrak X$ is a special formal $R$-scheme, any formal completion of $\mathfrak X$ is also a special formal $R$-scheme. For each special formal scheme $\mathfrak X$, its {\em reduction} $\mathfrak X_0$ is defined locally as $(\Spf A)_0:=\Spec A/J$.

In this article, a morphism between special formal $R$-schemes $\mathfrak f\colon \mathfrak Y\to \mathfrak X$ is an adic morphism of  special formal $R$-schemes. It induces a morphism $\mathfrak f_0\colon \mathfrak Y_0\to \mathfrak X_0$ of $k$-varieties at the reduction level. The category of special formal $R$-schemes admits fiber products and the assignment 
$$\mathfrak X\mapsto\mathfrak X_0$$ 
from the category of special formal $R$-schemes to the category of $k$-varieties is functorial. Furthermore, the natural closed immersion $\mathfrak X_0\to \mathfrak X$ is a homeomorphism. 	
\end{definition}

There is a functor of generic fibes, which associates to a special formal $R$-scheme a rigid $K$-variety. As explained in \cite[0.2.6]{Berthelot}, one first considers the affine case $\mathfrak X=\Spf A$, where $A$ is a special $R$-algebra. Denote by $J$ the largest ideal of definition of $A$ and consider for each $n\in \mathbb N^*$ the subalgebra $A\left[\varpi^{-1}J^n\right]$ of $A\otimes_RK$ generated by $A$ and $\varpi^{-1}J^n$. Let $B_n:=B_n(A)$ be the $J$-adic completion of $A[\varpi^{-1}J^n]$. Then we have the affinoid $K$-algebra $C_n:=B_n\otimes_RK$. The inclusion $J^{n+1}\subseteq J^n$ gives rise to a morphism of affinoid $K$-algebras $C_{n+1}\to  C_n$, which in turn induces an open embedding of affinoid $K$-spaces $\Sp(C_n)\to \Sp(C_{n+1})$. The {\it generic fiber} $\mathfrak X_{\eta}$ of $\mathfrak X$ is defined to be  
$$\mathfrak X_{\eta}=\bigcup_{n\in \mathbb N^*}\Sp(C_n).$$ 
Since this construction is functorial, we obtain the {\it generic fiber} of a special formal $R$-scheme $\mathfrak X$ by a glueing process. We call $\mathfrak X$ a {\em formal model} of $\mathfrak X_\eta$. The assignment $\mathfrak X\mapsto \mathfrak X_{\eta}$ is a functor from the category of special formal $R$-schemes to the category of separated rigid $K$-varieties. This functor commutes with fiber products.

 Let us look at morphisms of special formal schemes of the form $\Spf\left(R^{\prime}\right) \rightarrow \mathfrak{X}$, where $R \subset R^{\prime}$ is a finite extension of discrete valuation rings. Another such morphism $\Spf (R^{''})\rightarrow \mathfrak{X}$ is said to be {\em equivalent} to $\Spf\left(R^{\prime}\right) \rightarrow \mathfrak{X}$ if there exists a commutative diagram
\begin{displaymath}
\xymatrix@=3 em{
	\Spf\left(R^{\prime\prime\prime}\right) \ar[d]_{}\ar[r]&\Spf\left(R^{\prime}\right)\ar[d]_{}\\
	\Spf\left(R^{\prime \prime}\right)\ar[r]& \mathfrak{X}
}
\end{displaymath}
where $R\subset R^{\prime\prime\prime}$ is also a finite extension of discrete valuation rings. Then points of $\mathfrak{X}_\eta$ correspond bijectively with equivalence classes of such morphisms $\operatorname{Spf}\left(R^{\prime}\right) \rightarrow \mathfrak{X}$. 

There is a {\it specilization map} $\mathfrak X_{\eta}\to\mathfrak X$ defined as follows. For the affine case, the map $\sp\colon \Spf A_{\eta}\to \Spf A$ is defined as follows. Let $x$ be in $(\Spf A)_{\eta}$, and let $I\subseteq A\otimes_RK$ be the maximal ideal in $A\otimes_RK$ corresponding to $x$. Put $I'=I\cap A\subseteq A$. Then, by construction, $\sp(x)$ is the unique maximal ideal of $A$ containing $\varpi$ and $I'$. If the point $x$ corresponds to the equivalence class of $\varphi\colon \operatorname{Spf}\left(R^{\prime}\right) \rightarrow \mathfrak{X}$ as above then $\sp (x) =\varphi$. 

If $Z$ is a locally closed subscheme of $(\Spf A)_0$, $\sp^{-1}(Z)$ is an open rigid $K$-subvariety of $(\Spf A)_{\eta}$, which is canonically isomorphic to the generic fiber of the formal completion of $\Spf A$ along $Z$ (cf. \cite[Section 7.1]{deJ}). In general, the construction of the specialization map $\sp\colon \mathfrak X_{\eta} \to \mathfrak X$ can be generalized to any special formal $R$-scheme $\mathfrak X$ using a glueing process (see \cite{deJ}).

Let $G$ be a finite algebraic group over $k$. In this article we fix an {\em action} of $G$ of $\Spf R$, i.e.~ an adic morphism of formal schemes $G\times_k\Spf R\to \Spf R$ satisfying certain conditions of an action. Let $\mathfrak X$ be a formal $R$-scheme, with structural morphism $\mathfrak X\to \Spf R$ viewed as a morphism of formal $k$-scheme. A {\it $G$-action} on $\mathfrak X$ is a $G$-action on the formal $k$-scheme $\mathfrak X$ (with the $k$-scheme structure induced from $k\hookrightarrow R$) such that $\mathfrak f$ is a $G$-equivariant $k$-morphism. A $G$-action on $\mathfrak X$ is called {\it good} if any orbit of it is contained in an affine open formal subscheme of $\mathfrak X$. If $\mathfrak f\colon \mathfrak Y \to \mathfrak X$ is a $G$-equivariant morphism of special formal $R$-schemes, then its reduction $\mathfrak f_0\colon \mathfrak Y_0\to \mathfrak X_0$ is also $G$-equivariant.

\subsection{Order of top forms}
Let $\mathfrak X$ be a special formal scheme of pure relative dimension $d$. Let $R'$ be an extension of $R$ of ramification index $e$ with a fixed uniformizing parameter $\varpi'$, and denote by $K'$ its quotient field. For any $R'$-point $\psi$ of $\mathfrak X$, the {\em order} of a  differential form $\widetilde\omega$ in $\Omega_{\mathfrak X/R}^d(\mathfrak X)$ at $\psi$ is defined as follows. Since $\big(\gamma^*\Omega_{\mathfrak X/R}^d\big)/(\text{torsion})$ is a free $\mathcal O_{R'}$-module of rank one, we have either $\gamma^*\widetilde\omega=0$ or $\gamma^*\widetilde\omega=\alpha \varpi'^n$ for some nonzero $\alpha\in \mathcal O_{R'}$ and $n\in\mathbb N$. Then we define 
\begin{equation*}
\ord_{\widetilde\omega}(\psi)=
\begin{cases}
\infty & \ \text{if}\ \ \gamma^*\widetilde\omega=0\\
n & \ \text{if}\ \ \gamma^*\widetilde\omega=\alpha \varpi'^n.
\end{cases}
\end{equation*}
Consider the canonical injective morphism (cf. \cite[Sect. 7]{deJ} and \cite[Sect. 2.1]{Ni2})
\begin{align}\label{canonicalinjection}
\Phi\colon \Omega_{\mathfrak X/R}^d(\mathfrak X)\otimes_RK\to \Omega_{\mathfrak X_{\eta}/K}^d(\mathfrak X_{\eta})
\end{align} 
 which factors uniquely through the sheafification map $\Omega_{\mathfrak X/R}^d(\mathfrak X)\otimes_RK \to \big(\Omega_{\mathfrak X/R}^d\otimes_RK\big)(\mathfrak X)$. A form $\omega$ on $\mathfrak X_{\eta}$ which lies in $\mathrm{Im}(\Phi)$ is called {\it $\mathfrak X$-bounded}. Note that, if $\mathfrak X$ is {\em stft}, i.e.~ $(\varpi)$ is its largest ideal of definition, then $\Phi$ is an isormorphism as shown in \cite[Proposition 1.5]{BLR95}.  A {\it gauge} form $\omega$ on $\mathfrak X_{\eta}$ is a global section of the differential sheaf $\Omega_{\mathfrak X_{\eta}/K}^d$ such that it generates the sheaf at every point of $\mathfrak X_{\eta}$. 

For any $\mathfrak X$-bounded gauge form $\omega$ on $\mathfrak X_{\eta}$, there exist $\widetilde\omega\in \Omega_{\mathfrak X/R}^d(\mathfrak X)$ and $n\in \mathbb N$ such that $\omega=\varpi^{-n}\widetilde{\omega}$. Then the difference 
\begin{align}\label{order}
\ord_{\omega}(\psi):=\ord_{\widetilde\omega}(\psi)-e\cdot n
\end{align}
 is independent of the choice of $\widetilde{\omega}$ as seen in \cite[Definition 5.4]{Ni2}). 


\subsection{Motivic $G$-integral on stft schemes}
Let $\mathfrak X$ be {\em stft}, i.e.~ a special formal $R$-scheme whose largest ideal of definition is $(\varpi)$. Let $R_n:=R/\varpi^{n+1}$ and let $X_n$ be the $k$-variety defined as $(\mathfrak X, \mathcal O_{\mathfrak X}\otimes_R R_n)$. For any $k$-algebra $A$ we denote $L(A) = A$ if $R$ has equal characteristics and $L(A) = W(A)$ otherwise, where $W(A)$ is the ring of Witt vectors over $A$. In \cite{Gr}, Greenberg shows that the functor defined locally by
$$\Spec A\mapsto \Hom_{R_n}(\Spec\left(R_n\otimes_{L(k)}L(A)\right),X_n)$$ 
from the category of $k$-schemes to the category of sets is presented by a $k$-scheme $\Gr_n(X_n)$ of finite type such that, for any $k$-algebra $A$,
$$\Gr_n(X_n)(A)=X_n(R_n\otimes_{L(k)}L(A)).$$  
The varieties $\Gr_n(X_n)$ together with natural truncation maps form a projective system whose limit is denoted by $\Gr(\mathfrak X)$ and called the Greenberg scheme of $\mathfrak X$. Note that, by \cite[Proposition 3.7]{LN22}, any action of $G$ on $\mathfrak X$ induces actions of $G$ on $\Gr_n(X_n)$ and $\Gr(\mathfrak X)$ such that the natural morphism $\pi_n\colon \Gr(\mathfrak X)\to \Gr_n(X_n)$ is $G$-equivariant.  A subset of $\Gr(\mathfrak X)$ is called a G-invariant cylinder if it is equal to $\pi_n^{-1}(C)$ for some $G$-invariant constructible subset $C$ of  $\Gr_n(X_n)$. Let $\mathbf C_{\mathfrak X}^G$ be the set of $G$-invariant cylinders of $\Gr(\mathfrak X)$. It follows from \cite[Proposition 3.9]{LN22} that there exists a unique additive mapping
$$\mu_{\mathfrak X}^G\colon \mathbf C_{\mathfrak X}^G\to \mathscr M_{\mathfrak X_0}^G$$ 
such that for any $G$-invariant cylinder $\mathscr A$ of level $n$ of $\Gr(\mathfrak X)$,
$$\mu_{\mathfrak X}^G(\mathscr A)=[\pi_n(\mathscr A)\to \mathfrak X_0]\L^{-(n+1)d}.$$ On the other hand, for any field extension $F$ of $k$ there is a bijection, see  \cite[Sect. 4.1]{LS},
$$\Gr(\mathfrak X)(F)\cong \mathfrak X(R\otimes F).$$
Therefore the function $\ord_\omega$ in (\ref{order}) defines a $\mathbb Z$-value function 
$$\ord_\omega\colon \Gr(\mathfrak X)\setminus \Gr(\mathfrak X_{\sing})\to\mathbb Z\cup \{\infty\},$$
which give rise to $G$-integrable function $\L^{-\ord_{\varpi,\mathfrak X}(\omega)}$ in the following sense: for any $\mathscr A$ in $\mathbf C_{\mathfrak X}^G$ and any simple function $\alpha\colon \mathscr A\to \mathbb Z \cup \{\infty\}$, we say that $\L^{-\alpha}$ is {\it $G$-integrable}, if $\alpha$ takes only finitely many values in $\mathbb Z$ and if all the fibers of $\alpha$ are in $\mathbf C_{\mathfrak X}^G$. We define the {\it motivic $G$-integral} of $\omega$ on $\mathscr X$ to be
$$\int_{\mathfrak X}|\omega|:=\sum_{n\in\mathbb Z}\mu_{\mathfrak X}^G(\ord_{\omega}^{-1}(n))\L^{-n}\ \in \mathscr M_{\mathfrak X_0}^G.$$


\subsection{Motivic $G$-integral on special formal schemes}
\begin{definition}\label{G-dilatation}
Let $G$ be a smooth algebraic group over $k$. Let $\mathfrak X$ be a flat generically smooth special formal $R$-scheme endowed with a good $G$-action. Let $\mathfrak Z$ be a $G$-invariant  closed formal subscheme of $\mathfrak X_0$ defined by an ideal sheaf $\mathcal I$.  Let $\pi\colon \mathfrak Y\to \mathfrak X$ be the equivariant admissible blowup with center $\mathcal Z$ and let $\mathfrak U$ be the open formal subscheme of $\mathfrak Y$ where $\mathcal I\mathcal O_{\mathfrak Y}$ is generated by $\varpi$, the restriction $\pi\colon \mathfrak U\to\mathfrak X$ is called {\it the $G$-dilatation of $\mathfrak X$ with center $\mathfrak Z$}. Let $\pi\colon \mathfrak U\to\mathfrak X $ be the $G$-dilatation of $\mathfrak X$, i.e.~ the $G$-dilatation with center $\mathfrak X_0$. Notice that, in this case, if $\mathfrak X$ is covered by affine formal subschems $\Spf A$, then $\mathfrak U$ can be constructed by glueing affine open formal subschemes $\Spf B_1(A)$ with the notation as in Section \ref{section23}.  For any gauge form $\omega$ on $\mathfrak X_{\eta}$, we define 
\begin{align*}
\int_{\mathfrak X}|\omega|:= {\pi_0}_!\int_{\mathfrak U}|\pi_{\eta}^*\omega|\quad \text{in} \ \ \mathscr M_{\mathfrak X_0}^G
\end{align*}
and call it the {\it motivic $G$-integral} of $\omega$ on $\mathfrak X$.  If $\mathfrak X$ is a generically smooth special formal $R$-scheme endowed with a good adic $G$-action, we denote by $\mathfrak X^{\flat}$ its maximal flat closed subscheme (obtained by killing $\varpi$-torsion), and define the {\it motivic $G$-integral} of a gauge form $\omega$ on $\mathfrak X$ to be
\begin{align*}
\int_{\mathfrak X}|\omega|:= \int_{\mathfrak X^{\flat}}|\omega| \quad \text{in} \ \ \mathscr M_{\mathfrak X_0}^G.
\end{align*}	
\end{definition}

We list below several properties of the motivic $G$-integrals.
\begin{proposition}[Special $G$-equivariant change of variables formula]\label{changeofvariables}
Let $G$ be a smooth finite group scheme over $k$. Let $\mathfrak X$ and $\mathfrak Y$ be generically smooth special formal $R$-schemes endowed with good adic actions of $G$, and let $\mathfrak h\colon \mathfrak Y \to\mathfrak X$ be an adic $G$-equivariant morphism of formal $R$-schemes such that  the induced morphism $\mathfrak Y_{\eta} \to \mathfrak X_{\eta}$ is an open embedding and $\mathfrak Y_{\eta}(K^{sh})=\mathfrak X_{\eta}(K^{sh})$. If  $\omega$ is a gauge form on $\mathfrak X_{\eta}$, then 
$$\int_{\mathfrak X}|\omega|={\mathfrak h_0}_!\int_{\mathfrak Y}| \mathfrak h_{\eta}^*\omega|\quad \text{in} \ \ \mathscr M_{\mathfrak X_0}^G.$$
\end{proposition}
\begin{proposition}[Additivity of motivic integrals]\label{int-additive}
Let $\mathfrak X$ be a generically smooth special formal $R$-scheme endowed with a good adic action of $G$ and let $\omega$ be a gauge form on $\mathfrak X_\eta$. 
\begin{itemize}
\item[(i)] If $\{U_i, i \in I\}$ is a finite stratification of $\mathfrak X_0$ into $G$-invariant locally closed subsets, and $\mathfrak U_i$ is the formal completion of $\mathfrak X$ along $U_i$, then
$$\int_{\mathfrak X}|\omega|=\sum_{i\in I} \int_{\mathfrak U_i}|\omega| \quad \text{in} \ \ \mathscr M_{\mathfrak X_0}^G.$$
\item[(ii)] If $\{\mathfrak U_i, i \in I\}$ is a finite covering of $G$-invariant open subsets of $\mathfrak X$, then
$$\int_{\mathfrak X}|\omega|=\sum_{I'\subseteq I}(-1)^{|I'|-1} \int_{\mathfrak U_{I'}}|\omega|  \quad \text{in} \ \ \mathscr M_{\mathfrak X_0}^G,$$
where $\mathfrak U_{I'}=\cap_{i\in I'} \mathfrak U_i$.
\end{itemize}
Here the pushforward morphisms $\mathscr M_{U_{I'}}^G\to  \mathscr M_{\mathfrak X_0}^G$ are applied to the RHS in both statements.
\end{proposition}

\begin{proposition}[Motivic $G$-intgral of smooth formal $R$-schemes]\label{special-connected}
Let $\mathfrak X$ be a smooth special formal $R$-scheme of pure relative dimension $d$, which is endowed with a good adic $G$-action. Suppose that $\omega$ is an $\mathfrak X$-bounded gauge form on $\mathfrak X_{\eta}$. Denote by $\mathcal C(\mathfrak X_0)$ the set of all connected components of $\mathfrak X_0$. Then for each $C\in \mathcal C(\mathfrak X_0)$ the value $\ord_{\omega}(\psi)$ is independent of the choice of $\psi$ such that $\psi(0)\in C$ and denoted by $\ord_C(\omega)$. Assume that every $C\in \mathcal C(\mathfrak X_0)$ is $G$-invariant, then the identity
$$\int_{\mathfrak X}|\omega|=\L^{-d}\sum_{C\in \mathcal C(\mathfrak X_0)}\left[C\hookrightarrow \mathfrak X_0\right]\L^{-\ord_C(\omega)}$$
holds in $\mathscr M_{\mathfrak X_0}^G$. In particular, if $\mathfrak X_0$ is connected, then 
$$\int_{\mathfrak X}|\omega|=\L^{-(d+\ord_{\mathfrak X_0}(\omega)}) \text{ in }\mathscr M_{\mathfrak X_0}^G.$$
\end{proposition}
\begin{example}\label{example21} 
Let $R=k[\![\varpi]\!]$, $K=k(\!(\varpi)\!)$, $R(n)=k[\![\varpi^{1/n}]\!]$ and $K(n)=k(\!(\varpi^{1/n})\!)$ for $n\in \mathbb N^*$. Consider the formal schemes $\mathfrak X=\Spf(R[\![x_1,\dots,x_d]\!])$ and  $\mathfrak Y=\Spf(R\{x_1,\dots,x_d\})$ with $\mathfrak X_0=\Spec k$ and $\mathfrak Y_0=\mathbb A^n_k$. 
Observe that $\Spf(R(n)\{x\})$ and $\Sp(K(n)\{x\})$ are endowed with the good $\mu_n$-action induced by $(\xi,\varpi^{1/n})\mapsto \xi \varpi^{1/n}$. Clearly, $dx:=dx_1\wedge\ldots\wedge dx_d$ is a gauge form on $\Sp(K\{x\})$, its pullback $dx(n)$ via the natural morphism $\Sp(K(n)\{x\})\to \Sp(K\{x\})$ is still $dx$. Since the canonical isomorphism (\ref{canonicalinjection}) is given by $dx\otimes f\mapsto fdx$, the functions $\ord_{\varpi,\mathfrak X}(dx)$ and $\ord_{\varpi^{1/n},\mathfrak X(n)}(dx(n))$ are zero constant. Thus
$$\int_{\mathfrak X}|dx|=\L^{-d} \in \mathscr M_{k}^{\mu_1}, \int_{\mathfrak X(n)}|dx(n)|=\L^{-d} \in \mathscr M_{k}^{\mu_n}.$$
Similarly, one can show that
$$\int_{\mathfrak Y}|dx|=\L^{-d} \in \mathscr M_{\mathbb A_k^d}^{\mu_1}, \int_{\mathfrak Y(n)}|dx(n)|=\L^{-d} \in \mathscr M_{\mathbb A_k^d}^{\mu_n}.$$
Now write $R\{\frac{x} {\varpi^p}\}$ for $R\{x,y\}/(x-\varpi^py)$ for new variables $y=(y_1,\ldots,y_{d})$ and a positive integer $p$. Consider the formal $R$-schemes $\mathfrak U=\Spf(R\{\frac{x} {\varpi^p}\})$ and $\mathfrak U(n)=\Spf(R(n)\{\frac{x} {\varpi^p}\})$. Consider the gauge form $dx$ on $\mathfrak U_{\eta}=\Sp(K\{\frac{x} {\varpi^p}\})$. Then its pullback $dx(n)$ via the natural morphism $\Sp(K(n)\{\frac{x} {\varpi^p}\})\to \Sp(K\{\frac{x} {\varpi^p}\})$ is nothing else than $dx$. Since $\varpi^{dp}d\frac{x} {\varpi^p} \otimes 1 \mapsto dx$ via (\ref{canonicalinjection}), the function $\ord_{\varpi,\mathfrak U}(dx)$ is the constant function $dp$ while $\ord_{\varpi^{1/n},\mathfrak U(n)}(dx(n))$ is the constant function $ndp$, hence
$$\int_{\mathfrak U}|dx|=\L^{-d(p+1)} \in \mathscr M_{\mathbb A_k^d}^{\mu_1},\ \ \int_{\mathfrak U(n)}|dx(n)|=\L^{-d(np+1)} \in \mathscr M_{\mathbb A_k^d}^{\mu_n}.$$ 
\end{example}
\subsection{Monodromic volume Poincar\'e series and motivic volumes}\label{motvol}
A special formal $R$-scheme $\mathfrak X$ is called {\it regular} if $\mathcal O_{\mathfrak X,x}$ is regular for every $x\in \mathfrak X$. By \cite[Definition 2.33]{Ni2}, a closed formal subscheme $\mathfrak E$ of a purely relatively $d$-dimensional special formal $R$-scheme $\mathfrak X$ is called a {\it strict normal crossings divisor} if, for every $x$ in $\mathfrak X$, there exists a regular system of local parameters $(x_0,\dots,x_d)$ in $\mathcal O_{\mathfrak X,x}$ such that the ideal defining $\mathfrak E$ at $x$ is locally generated by $\prod_{i=0}^dx_i^{N_i}$ for some $N_i\in \mathbb N$, $0\leq i\leq d$, and such that the irreducible components of $\mathfrak E$ are regular (see \cite[Section 2.4]{Ni2} for definition of irreducibility). If $\mathfrak E_i$ is an irreducible component of $\mathfrak E$ which is defined locally by the ideal $x_i$, it is a fact that $N_i$ is constant when $x$ varies on $\mathfrak E_i$. 
Then we have $\mathfrak E=\sum_{i\in S}N_i\mathfrak E_i$, where $\mathfrak E_i$'s are irreducible components of $\mathfrak E$. The divisor $\mathfrak E$ is called a {\it tame strict normal crossings divisor} if $N_i$ is prime to the characteristic exponent of $k$ for every $i$. Any special formal $R$-scheme $\mathfrak X$ is said to {\it have tame strict normal crossings} if $\mathfrak X$ is regular and $\mathfrak X_s$ is a tame strict normal crossings divisor. 

Fix $i \in S$ and let $x$ be a point of $E_{i}$. We define the localization of $\mathcal{O}_{\mathfrak{X}, x}$ at the generic point corresponding to $\mathfrak{E}_{i}$. Then by \cite[Lemma 7.3]{Ni2}, the $\mathcal{O}_{\mathfrak{X}, \mathfrak{E}_{i}, x}$-module $\Omega_{\mathfrak{X}, \mathfrak{E}_{i}, x}$ is free of rank one. For any
$$
\omega \in \Omega_{\mathfrak{X} / R}^{m}(\mathfrak{X}) /(\pi-\text { torsion })
$$
we define the order of $\omega$ along $\mathfrak{E}_{i}$ at $x$ as the length of the $\mathcal{O}_{\mathfrak{X}, \mathfrak{E}_{i}, x}$-module $\Omega_{\mathfrak{X}, \mathfrak{E}_{i}, x} /$ $\left(\mathcal{O}_{\mathfrak{X}, \mathfrak{E}_{i}, x} \cdot \omega\right)$, and we denote it by ord $_{\mathfrak{E}_{i}, x} \omega$.

If $\omega$ is a $\mathfrak{X}$-bounded $m$-form on $\mathfrak{X}_{\eta}$, there exists an integer $a \geq 0$ and an affine open formal subscheme $\mathfrak{U}$ of $\mathfrak{X}$ containing $x$, such that $\pi^{a} \omega$ belongs to
$$
\Omega_{\mathfrak{X} / R}^{m}(\mathfrak{U}) /(\pi-\text { torsion }) \subset \Omega_{\mathfrak{X} / R}^{m}(\mathfrak{U}) \otimes_{R} K
$$
We define the order of $\omega$ along $\mathfrak{E}_{i}$ at $x$ as
$$
\operatorname{ord}_{\mathfrak{E}_{i}, x} \omega:=\operatorname{ord}_{\mathfrak{E}_{i}, x}\left(\pi^{a} \omega\right)-a N_{i}
$$
which is dependent of $a$ as well as $x\in E_i$ and is denoted by $\operatorname{ord}_{\mathfrak{E}_{i}} \omega$.

We now study the $\hat\mu$-equivariant setting of volume Poincar\'e series and motivic volume of special formal $R$-schemes. Note that their older version (without action) was performed early in \cite[Sect. 7]{Ni2}.

For  $n\in \mathbb N^*$, we put $R(n)=R[\tau]/(\tau^n-\varpi)$ and $K(n)=K[\tau]/(\tau^n-\varpi)$. For any formal $R$-scheme $\mathfrak X$, we define its ramifications as follows: $\mathfrak X(n)=\mathfrak X\times_RR(n)$, $\mathfrak X_{\eta}(n)=\mathfrak X_{\eta}\times_KK(n)$. If $\omega$ is a gauge form on $\mathfrak X_{\eta}$, let $\omega(n)$ be its pullback via the natural morphism $\mathfrak X_{\eta}(n)\to \mathfrak X_{\eta}$, which is a gauge form on $\mathfrak X_{\eta}(n)$.
Let $\mathfrak X$ be a formal $R$ scheme and $n$ in $\mathbb N^*$. Then there is a natural good adic $\mu_n$-action on both $\Spf R(n)$ and $\mathfrak X(n)$ which is induced from the ring homomorphism $R(n)\to k[\xi]/(\xi^n-1)\otimes_kR(n)$ given by $\tau\mapsto \xi\otimes \tau$. Moreover, the structural morphism of the formal $\Spf R(n)$-scheme $\mathfrak X(n)$ is $\mu_n$-equivariant.

Remark that if $\mathfrak X$ is a generically smooth special formal $R$-scheme and $n\in \mathbb N^*$, then $\mathfrak X(n)$ is a generically smooth special formal $R(n)$-scheme.

\begin{definition}\label{Poincareseries}
Let $\mathfrak X$ be a generically smooth special formal $R$-scheme, and let $\omega$ be a gauge form on $\mathfrak X_{\eta}$. The formal power series
$$
P(\mathfrak X,\omega; T):=\sum_{n\geq 1}\Big(\int_{\mathfrak X(n)}|\omega(n)|\Big)T^n\ \in \mathscr M_{\mathfrak X_0}^{\hat\mu}[\![T]\!]
$$
\end{definition}

\begin{definition}\label{resolutionofsingularities}
Let $\mathfrak X$ be a generically smooth flat special formal $R$-scheme of pure relative dimension $d$. Assume that $\mathfrak X$ admits a resolution of singularities $\mathfrak h\colon \mathfrak Y \to \mathfrak X$. A {\it resolution of singularities} of $\mathfrak X$ is a proper morphism of flat special formal $R$-schemes $\mathfrak h\colon\mathfrak Y\to\mathfrak X$, such that $\mathfrak h_{\eta}$ is an isomorphism and $\mathfrak Y$ is regular with $\mathfrak Y_s$ being a strict normal crossings divisor. The resolution of singularities $\mathfrak h$ is said to be {\it tame} if $\mathfrak Y_s$ is a tame strict normal crossings divisor. Note that if the base field $k$ has characteristic zero, then any generically smooth flat special formal $R$-scheme $\mathfrak X$ admits a resolution of singularities (cf. \cite{Tem}). Let $\mathfrak E_i$, $i\in S$, be the irreducible components of $(\mathfrak Y_s)_{\mathrm{red}}$. Let $N_i$ be the multiplicity of $\mathfrak E_i$ in $\mathfrak Y_s$. Put 
$$E_i:=(\mathfrak E_i)_0$$ 
for $i\in S$ and
$$\mathfrak E_I:=\bigcap_{i\in I}\mathfrak E_i,\quad E_I:=\bigcap_{i\in I}E_i, \quad E_I^{\circ}:=E_I\setminus\bigcup_{j\not\in I}E_j$$ 
for any nonempty subset $I\subseteq S$. We can check that $\mathfrak E_I$ is regular and $E_I=(\mathfrak E_I)_0$. Let $\{\mathfrak U=\Spf \mathcal A\}$ be a covering of $\mathfrak Y$ by affine open subschemes. If $\mathfrak U_0\cap E_I^{\circ}\not=\emptyset$, the composition $\mathfrak f\circ\mathfrak h\colon \mathfrak U\to \Spf R$ is defined on the ring level by 
$$\varpi\mapsto u\prod_{i\in I}y_i^{N_i},$$ 
where $u$ is a unit in $(y_i)_{i\in I}$ and $y_i$ is a local coordinate defining $E_i$. Put $N_I:=\gcd(N_i)_{i\in I}$. Then we can construct as in \cite{DL5, Ni2} an unramified Galois covering $\widetilde{E}_I^{\circ}\to E_I^{\circ}$ with Galois group $\mu_{N_I}$ which is given over $\mathfrak U_0\cap E_I^{\circ}$ as the reduction of 
$$\Spf\left(\mathcal A[z]/u z^{N_I}-1\right).$$
Notice that $\widetilde{E}_I^{\circ}$ is endowed with a natural good adic $\mu_{N_I}$-action over $E_I^{\circ}$ obtained by multiplying the $z$-coordinate with elements of $\mu_{N_I}$. Denote by $\big[\widetilde E_I^{\circ}\big]$ the class of the $\mu_{N_I}$-equivariant morphism 
$$\widetilde E_I^{\circ}\to E_I^{\circ}\to \mathfrak X_0$$ 
in the ring $\mathscr M_{\mathfrak X_0}^{\mu_{N_I}}$.
\end{definition}

\begin{theorem}\label{int_Xm}
Let $\mathfrak X$ be a generically smooth flat special formal $R$-scheme of pure relative dimension $d$. Suppose that we are given a tame resolution of singularities $\mathfrak h\colon \mathfrak Y \to \mathfrak X$ with $\mathfrak Y_s=\sum_{i\in S}N_i\mathfrak E_i$ and an $\mathfrak X$-bounded gauge form $\omega$ on $\mathfrak X_{\eta}$ with order $\alpha_i:=\ord_{\mathfrak E_i}(\mathfrak h_{\eta}^*\omega)$ for every $i\in S$. If $n\in \mathbb N^*$ is prime to the characteristic exponent of $k$, then the identity
$$
\int_{\mathfrak X(n)}|\omega(n)|=\L^{-d}\sum_{\emptyset\not=I\subseteq S}(\L-1)^{|I|-1}\big[\widetilde{E}_I^{\circ}\big]\left(\sum_{\begin{smallmatrix} k_i\geq 1, i\in I\\ \sum_{i\in I}k_iN_i=n \end{smallmatrix}}\L^{-\sum_{i\in I}k_i\alpha_i}\right).
$$
holds in $\mathscr M_{\mathfrak X_0}^{\mu_n}$.
\end{theorem}

\begin{corollary}\label{poincare}
Suppose that the base field $k$ has characteristic zero. Let $\mathfrak X$ be a generically smooth flat special formal $R$-scheme of relative dimension $d$. Let $\mathfrak h\colon \mathfrak Y \to \mathfrak X$ be a resolution of singularities with $\mathfrak Y_s=\sum_{i\in S}N_i\mathfrak E_i$. Suppose that $\omega$ is an $\mathfrak X$-bounded gauge form on $\mathfrak X_{\eta}$ with $\alpha_i:=\ord_{\mathfrak E_i}(\mathfrak h_{\eta}^*\omega)$ for every $i\in S$. Then 
$$P(\mathfrak X,\omega;T)=\L^{-d}\sum_{\emptyset\not=I\subseteq S}(\L-1)^{|I|-1}\big[\widetilde{E}_I^{\circ}\big]\prod_{i\in I}\frac{\L^{-\alpha_i}T^{N_i}} {1-\L^{-\alpha_i}T^{N_i}}.$$
\end{corollary}

We deduce from Corollary \ref{poincare} that the limit 
$$\lim\limits_{T\to\infty}P(\mathfrak X,\omega;T)=-\L^{-d}\sum_{\emptyset\not=I\subseteq S}(1-\L)^{|I|-1}\big[\widetilde{E}_I^{\circ}\big] \in \mathscr M_{\mathfrak X_0}^{\hat\mu}$$
is independent of the choice of the $\mathfrak X$-bounded gauge form $\omega$. However, it depends on the choice of the uniformizer (see \cite[Rem. 7.40]{Ni2}).
\begin{definition}\label{defMVformal}
 Suppose that $k$ has characteristic zero. Let $\mathfrak{X}$ be a generically smooth special formal $R$-scheme of pure dimension $d$. Assume that $\mathfrak{X}_\eta$ admits an $\mathfrak{X}$-bounded gauge form. Then we define
$$\MV\left(\mathfrak{X}\right):=-\L^{d}\lim\limits_{T\to\infty}P(\mathfrak X,\omega;T).$$
In general, take a resolution of singularities $h\colon \mathfrak{Y} \rightarrow \mathfrak{X}$ of $\mathfrak{X}$ and take a finite open cover $\left\{\mathfrak{U}_{i}\right\}_{i \in I}$ of $\mathfrak{Y}$ such that $\mathfrak{U}_{i}$ has pure relative dimension and $\left(\mathfrak{U}_{i}\right)_{\eta}$ admits a $\mathfrak{U}_{i}$-bounded gauge form for each $i$. Then, the value
$$
\MV\left(\mathfrak{X}\right)=\sum_{\emptyset \neq J \subset I}(-1)^{|J|+1} \MV\left(\cap_{i \in J} \mathfrak{U}_{i} \right) \in \mathscr{M}_{\mathfrak{X}_{0}}^{\hat\mu}
$$
only depends on $\mathfrak{X}$ and is called the {\em motivic volume of $\mathfrak{X}$} (cf. \cite[7.38,7.39]{Ni2}).
\end{definition}

\begin{proposition}[Additivity of $\MV$]\label{MV-additive}
Suppose that $k$ has characteristic zero. Let $\mathfrak X$ be a generically smooth special formal $R$-scheme. The following hold.
\begin{itemize}
\item[(i)] If $\{U_i, i \in Q\}$ is a finite stratification of $\mathfrak X_0$ into locally closed subsets, and $\mathfrak U_i$ is the formal completion of $\mathfrak X$ along $U_i$, then
$$\MV(\mathfrak X)=\sum_{i\in Q} \MV(\mathfrak U_i).$$
\item[(ii)] If $\{\mathfrak U_i, i \in Q\}$ is a finite open covering of $\mathfrak X$, then by putting $\mathfrak U_{I}=\bigcap_{i\in I} \mathfrak U_i$, we have
$$\MV(\mathfrak X)=\sum_{\emptyset\not=I\subseteq Q}(-1)^{|I|-1} \MV(\mathfrak U_{I}).$$
\end{itemize}
\end{proposition}
\begin{example}\label{example22} 
Let $R=k[\![\varpi]\!]$, $K=k(\!(\varpi)\!)$, $R(n)=k[\![\varpi^{1/n}]\!]$ and $K(n)=k(\!(\varpi^{1/n})\!)$ for $n\in \mathbb N^*$. Consider the formal schemes $\mathfrak X=\Spf(R[\![x_1,\dots,x_d]\!])$, $\mathfrak Y=\Spf(R\{x_1,\dots,x_d\})$ and  $\mathfrak U=\Spf(R\{\frac{x} {\varpi^p}\})$  with $R\{\frac{x} {\varpi^p}\}:=R\{x,y\}/(x-\varpi^py)$ for new variables $y=(y_1,\ldots,y_{d})$. 
Using the computations in Example \ref{example21} we obtain
$$\MV(\mathfrak X)=1\in \mathscr M_{k}^{\hat\mu},\ \MV(\mathfrak Y)=\MV(\mathfrak U)=1\in \mathscr M_{\mathbb A_k^d}^{\hat\mu}.$$
\end{example}

\subsection{Motivic zeta functions and motivic nearby cycles of formal power series}\label{Mot_series}
Consider the mixed formal power series $R$-algebra $R\{x\}[\![y]\!]$, with $x=(x_1,\dots, x_m)$ and $y=(y_1,\dots,y_{m'})$. Let $d=m+m'$. Let $f$ be in $k\{x\}[\![y]\!]$ such that $f(x,0)$ is not a non-zero constant, and let $\mathfrak X(f)$ be the formal completion of $\Spf(k\{x\}[\![y]\!])$ along $(f)$. Then $\mathfrak X(f)$ is a genericall smooth special formal $R$-scheme of pure relative dimension $d-1$, with structural morphism defined by $\varpi \mapsto f$. Moreover, it follows from  \cite[Lemma 4.29]{LN22} that 
\begin{align}\label{modeloff}
\mathfrak X(f)\cong \Spf\left( R\{x\}[\![y]\!]/(f-\varpi)\right)
\end{align}
 By \cite[Sect. 4]{AJP}, we can see that $\mathfrak X(f)$ is a formal scheme of pseudo-finite type over $k$, the sheaf of continuous differential form $\Omega^i_{\mathfrak X(f)/k}$ is coherent for any $i$, and that there exists a morphism of coherent $\mathcal O_{\mathfrak X(f)}$-modules $d\varpi \wedge(.) \colon \Omega_{\mathfrak X(f)/R}^{d-1}\to \Omega_{\mathfrak X(f)/k}^d$ defined by taking the exterior product with the differential $df$. Taking the ``rig" functor (\cite[Sect. 7]{deJ}) we get a morphism of coherent $\mathcal O_{\mathfrak X(f)_{\eta}}$-modules 
$$d\varpi \wedge(.) \colon \Omega_{\mathfrak X(f)_{\eta}/K}^{d-1}\to (\Omega_{\mathfrak X(f)/k}^d)_{\rig},$$
which is an isomorphism due to \cite[Prop. 7.19]{Ni2}.  In this section we fix the gauge form $\omega$ on $\mathfrak X(f)$ defined as $\omega=dx_1\wedge\cdots\wedge dx_m\wedge dy_1\wedge\cdots\wedge dy_{m'}$ and denote by $\omega/df$ the inverse image of $\omega$ under $d\varpi \wedge(.)$ and call it the {\em standard Gelfand-Leray form}.

Let $\mathfrak h\colon \mathfrak Y\to \mathfrak X(f)$ be a tame resolution of singularities of $\mathfrak X(f)$. Assume that the data of $\mathfrak Y$ are given as in the setting before Theorem \ref{int_Xm} and that $K_{\mathfrak Y/\mathfrak X(f)}=\sum_{i\in S}(\nu_i-1)\mathfrak E_i$. Using the same argument in the proof of \cite[Lem. 7.30]{Ni2} we get $\ord_{\mathfrak E_i}\mathfrak h_{\eta}^*(\omega/df)=\nu_i-N_i$ for all $i\in S$. Note that these numbers do not depend on $\omega$. Similarly as in the proof of Theorem \ref{int_Xm} we have the following result.

\begin{proposition}\label{mzf-fseries}
With the previous notation and hypotheses, if $n\in \mathbb N^*$ is prime to the characteristic exponent of $k$ and not $\mathfrak Y_s$-linear, then the identity
$$
\int_{\mathfrak X(f)(n)}|(\omega/df)(n)|=\L^{n+1-d}\sum_{\emptyset\not=I\subseteq S}(\L-1)^{|I|-1}\big[\widetilde{E}_I^{\circ}\big]\left(\sum_{\begin{smallmatrix} k_i\geq 1, i\in I\\ \sum_{i\in I}k_iN_i=n \end{smallmatrix}}\L^{\sum_{i\in I}k_i(N_i-\nu_i)}\right).
$$
holds in $\mathscr M_{\mathfrak X(f)_0}^{\mu_n}$. If, in addition, $k$ has characteristic zero, then 
$$P(\mathfrak X(f),\omega/df;T)=\L^{-(d-1)} \frac{\L T}{1-\L T}\ast\sum_{\emptyset\not=I\subseteq S}(\L-1)^{|I|-1}\big[\widetilde{E}_I^{\circ}\big]\prod_{i\in I}\frac{\L^{-\nu_i}T^{N_i}} {1-\L^{-\nu_i}T^{N_i}},$$
where $\ast$ is the Hadamard product of formal power series in $\mathscr M_{\mathfrak X(f)_0}^{\hat\mu}[\![T]\!]$. Moreover, 
$$\MV(\mathfrak X(f);\widehat{K^s})=\sum_{\emptyset\not=I\subseteq S}(1-\L)^{|I|-1}\big[\widetilde{E}_I^{\circ}\big] \in \mathscr M_{\mathfrak X(f)_0}^{\hat\mu}.$$
\end{proposition}
\begin{definition}\label{neasrbycycle}
Let $k$ be a field of characteristic zero. Let $f$ be in $k\{x\}[\![y]\!]$ such that $f(x,0)$ is non-constant. Let $\x$ be a closed point in $\mathfrak X(f)_0$. The {\it motivc zeta function} $Z_f(T)$ of $f$ and the {\it local motivic zeta function} $Z_{f,\x}(T)$ of $f$ at $\x$ are defined as follows  
$$Z_f(T):=\L^{d-1}P(\mathfrak X(f),\omega/df;T),\ \ Z_{f,\x}(T):=\L^{d-1}P(\widehat{\mathfrak X(f)_{/\x}},\omega/df;T).$$ 
The {\it motivic nearby cycles } $\mathscr S_f$ of $f$ and the {\it motivic Milnor fiber} $\mathscr S_{f,\x}$ of $f$ at $\x$ are defined as
$$\mathscr S_f:=\MV(\mathfrak X(f))\in \mathscr M_{\mathfrak X(f)_0}^{\hat\mu},\ \ \mathscr S_{f,\x}:=\MV(\widehat{\mathfrak X(f)_{/\x}})\in \mathscr M_k^{\hat\mu}.$$ 
\end{definition}
\begin{proposition}\label{contactinvariant}
Let $f,g\in k\{x\}]\!]y]\!]$ be  two series such that $f=ug$ for some unit $u\in k\{x\}[\![y]\!]$ which admits at least an $n$-th root for all $n\geq 1$. Then 
$$\mathscr S_f=\mathscr S_{g}$$
in   $\mathscr M_{X_0}^{\hat{\mu}}$, where $X_0=\Spec\ k[x]/\left(f(x,0)\right)$.
\end{proposition}
\begin{proof}
Let us denote by $\mathfrak X$ the formal completion of the $k$-scheme $\Spf\ k\{x\}]\!]y]\!]$ along the ideal $(f)=(g)$. Let $\mathfrak X(f)$ and $\mathfrak X(g)$ be the formal $R$-scheme associated to $f$ and $g$ respectively. Note that, $\mathfrak X=\mathfrak X(f)=\mathfrak X(g)$ as formal $k$-schemes. Let $\mathfrak Y(f)\to\mathfrak X(f)$ be a resolution of singularities of $\mathfrak X(f)$. Then $\mathfrak Y(g):=\mathfrak Y(f)\to \mathfrak X\overset{g}{\to} \Spf R$ is also a resolution of formal $R$-schemes $\mathfrak X(g)$ respectively. It follows from (\ref{modeloff}) that
$$\mathfrak X(f)_s\cong \Spf R\{x\}[\![y]\!]/(f) \cong \Spf R\{x\}[\![y]\!]/(g)\cong \mathfrak X(g)_s.$$
Therefore $$\mathfrak Y(f)_s\cong \mathfrak Y(g)_s.$$
We now use the notation as in Definition \ref{Poincareseries} with $\tilde{E}^{\circ}_I$ for $\mathfrak Y(f)$ and $\tilde{F}^{\circ}_I$ for $\mathfrak Y(g)$. Then for each subset $I\subset S$, there is an isomorphism $\tilde{E}^{\circ}_I\cong \tilde{F}^{\circ}_I$ defined by 
$$(z,y)\mapsto (z.\xi,y),$$
where $\xi$ is induced from an $m_I$-th root of $u$. Hence $\MV\left(\mathfrak X(f)\right)=\MV\left(\mathfrak X(g)\right)$ according to Corollary \ref{mzf-fseries}.
\end{proof}


\section{Rigid analytic geometry and Motivic volume}\label{rig}
In this section we assume that the field $k$ has characeristic zero.
\subsection{Motivic integrals on rigid analytic varieties}
The motivic integral of a differential form of maximal degree $\omega$ on a quasi-compact smooth rigid variety $X$ was already defined by Loeser-Sebag \cite{LS}, and inspired by it, Nicaise-Sebag \cite{NS2} extended the notion to bounded smooth rigid varieties. In the present section we generalize the notion of motivic integral $\int_X|\omega|$ of a gauge form on special rigid variety $X$ in the equivariant setting. 

\begin{definition}[Special rigid varieties]
A rigid $K$-variety $X$ is called {\em (affine) special} if it admits a formal model which is an (affine) special formal $R$-scheme. If $X$ is affine, there exists a special $R$-algebra $A$ such that $X=(\Spf A)_{\eta}$; we then call $A$ a {\em coordinate ring} of $X$. 
\end{definition}

Let $G$ be a finite group scheme over $k$. Let $(\mathfrak X,\theta)$ be a pair consisting of a formal model $\mathfrak X$ of $X$ and a good $G$-action $\theta$ on $\mathfrak X$, we call $(\mathfrak X,\theta)$ a {\it $G$-pair} of $X$. Two $G$-pairs $(\mathfrak X,\theta)$ and $(\mathfrak X',\theta')$ of $X$ are {\it equivalent} if there exist a formal $R$-scheme $\mathfrak X''$ endowed with a good $G$-action and two $G$-equivariant morphisms $\mathfrak X''\to \mathfrak X'$ and $\mathfrak X''\to \mathfrak X$ such that the induced morphisms morphism $\mathfrak X''_\eta\to \mathfrak X'_\eta$ and $\mathfrak X''_\eta\to \mathfrak X_\eta$ are open embedding satisfying $\mathfrak X''_{\eta}(K^{sh})=X(K^{sh})$. A {\it good action} of $G$ (or {\it good $G$-action}) on a special rigid $K$-variety $X$ is given by an equivalence class of $G$-pairs $(\mathfrak X,\theta)$ of $X$. A morphism of special $K$-varieties $Y\to X$ is called {\it $G$-equivariant} if there exist representatives $(\mathfrak X,\theta)$, $(\mathfrak Y,\tau)$ and a $G$-equivariant morphism $\mathfrak Y\to \mathfrak X$ such that the induced morphism $\mathfrak Y_\eta\to \mathfrak X_\eta$ is the morphism $Y\to X$. 

\begin{lemma-def}\label{Rig-G-integral}
Let $X$ be a smooth special rigid $K$-variety endowed with a good $G$-action. Let $(\mathfrak X,\theta)$ and $(\mathfrak X',\theta')$ be two equivalent $G$-pairs of $X$. If $\omega$ is a gauge form on $X$, then
$$\int_{\mathfrak X_0}\int_{\mathfrak X}|\omega|=\int_{\mathfrak X'_0}\int_{\mathfrak X'}|\omega|\ \in \mathscr M_{k}^G.$$
We call this quantity the {\bf (motivic) $G$-integral} of $\omega$ on $X$ and denoted by $\int_{X}|\omega|$.
\end{lemma-def}

\begin{proof}
Since $(\mathfrak X,\theta)$ and $(\mathfrak X',\theta')$ are equivalent, there exist a pair $(\mathfrak X'',\theta'')$ and two $G$-equivariant morphisms $h\colon\mathfrak X''\to \mathfrak X$ and $h'\colon \mathfrak X''\to \mathfrak X'$ such that the induced morphisms $\mathfrak X''_\eta\to \mathfrak X'_\eta$ and $\mathfrak X''_\eta\to \mathfrak X_\eta$ are open embedding satisfying $\mathfrak X''_{\eta}(K^{sh})=X(K^{sh})$. By the special $G$-equivariant change of variables formula (cf. Proposition \ref{changeofvariables}) we have
$$ \int_{\mathfrak X}|\omega|={h_0}_!\int_{\mathfrak X''}|\omega| \ \in \mathscr M_{\mathfrak X_0}^G$$
and
$$ \int_{\mathfrak X'}|\omega|={h'_0}_!\int_{\mathfrak X''}|\omega| \ \in \mathscr M_{\mathfrak X'_0}^G.$$
Hence, it follows from the fact $\int_{\mathfrak X_0}\circ {h_0}_!=\int_{\mathfrak X''_0}=\int_{\mathfrak X'_0}\circ {h'_0}_!$ that
$$\int_{\mathfrak X_0}\int_{\mathfrak X}|\omega|=\int_{\mathfrak X'_0}\int_{\mathfrak X'}|\omega|\ \in \mathscr M_{k}^G.$$
\end{proof}

\begin{lemma}\label{openimmersion}
Let $U\to X$ be a $G$-equivariant open immersion of smooth special rigid $K$-varieties such that $U(K^{sh})=X(K^{sh})$. Then, for any gauge forms $\omega$ on $X$ and any $m\in \mathbb N$, the identity
$$\int_{U}|\omega|=\int_{X}|\omega|$$
holds in $\mathscr M_{k}^{G}$.
\end{lemma}
\begin{proof}
Let $\mathfrak U$ be a formal model of $U$ and by $\mathfrak Z$ its dilatation. Let $Z$ be the generic fiber of $\mathfrak Z$ which is quasi-compact. Then by definition, 
$$\int_{Z(n)}|\omega(n)|=\int_{\mathfrak Z_\eta(n)}|\omega(n)|=\int_{X(n)}|\omega(n)|.$$
Moreover, for any formal model $\mathfrak X$ of $X$, using the isomorphism 
$$\lim_{\text{models } \mathfrak Z \text{ of } Z} \operatorname{Mor}(\mathfrak Z,\mathfrak X)\to \operatorname{Mor}(Z,X)$$
in \cite[(7.1.7.1)]{deJ}, we obtain another model $\mathfrak Z'$ of $Z$ and a morphism $h\colon \mathfrak Z'\to \mathfrak X$ such that $h_\eta$ is the inclusion $Z\to X$. Hence
$$\int_{Z(n)}|\omega(n)|=\int_{\mathfrak Z'_\eta(n)}|\omega(n)|=\int_{X(n)}|\omega(n)|,$$
which completes the lemma.
\end{proof}
\begin{remark}
Assume that the action of $G$ on $X$ is trivial. Then the motivic integral of $X$ defined in \cite[Def. 5.10]{NS2} (see also \cite[Prop. 4.8]{Ni2}, \cite[Thm.-Def. 4.1.2]{LS}) is nothing but the image of the $G$-integral of $X$ under the forgetful morphism
$$\mathscr M_{k}^G\to \mathscr M_{k}.$$
\end{remark}

\begin{lemma}\label{productofint}
Let $X$ and $X'$ be two smooth special rigid $K$-varieties endowed with good $G$-actions. If $\omega$ and $\omega'$ are gauge forms on $X$ and $X'$ respectively, then the identity
$$\int_{X\times X'}|\omega\otimes \omega'|=\int_{X}|\omega|\cdot \int_{X'}|\omega'|$$
holds in $\mathscr M_{k}^G$. Here, the $G$-action on $X\times X'$ is the diagonal action.
\end{lemma}

\begin{proof}
Let $(\mathfrak X,\theta)$ and $(\mathfrak X',\theta')$ be $G$-pairs of $X$ and $X'$ respectively. Let $\pi\colon\mathfrak U\to \mathfrak X$ (resp. $\pi'\colon\mathfrak U'\to \mathfrak X'$) be the $G$-dilataion of $\mathfrak X$ (resp. $\mathfrak X'$). Then $\pi\times\pi'\colon\mathfrak U\times\mathfrak U'\to \mathfrak X\times\mathfrak X'$ is the $G$-dilatation of $\mathfrak X\times\mathfrak X'$. Let $U$ and $U'$ be the generic fibers of $\mathfrak U$ and $\mathfrak U'$, respectively. By Lemma-Definition \ref{Rig-G-integral}, we have identities
$$\int_{X}|\omega|=\int_{U}|\omega|,\ \int_{X'}|\omega'|=\int_{U'}|\omega'|$$
and
$$\int_{X\times X'}|\omega\otimes \omega'|=\int_{U\times U'}|\omega\otimes \omega'|$$
in the ring $\mathscr M_{k}^G$. Note that $U$ and $U'$ are quasi-compact rigid $K$-varieties, applying the proof of \cite[Prop. 4.1.5]{LS}, we obtain the identity 
$$\int_{U\times U'}|\omega\otimes \omega'|=\int_{U}|\omega|\cdot \int_{U'}|\omega'|$$
and hence the lemma.
\end{proof}

In order to define the notion of motivic volume for general special rigid $K$-varieties we need to prove that any two models of a special rigid $K$-variety $X$ can be dominated by another model. In the case where $X$ is quasi-compact, i.e. $X$ admits an admissible covering of affinoid varieties, this is proved to be true in \cite{BLR93}. For a proof of the claim in general, we need to apply the  descent theory developed by  de Jong  in \cite[$\S$7.5]{deJ}.
\begin{theorem}[de Jong's descent theory]\label{descent}
Let $X$ be a  special $K$-rigid variety with a model $\mathfrak X$. Consider a stratification
$$\mathfrak X_0=\sqcup_{i\in I} V_i$$ of $\mathfrak X_0$ into finitely many locally closed subvarieties $V_i$ of $\mathfrak X_0$. Let $\mathfrak V_i$ be the formal completion of $\mathfrak X$ along $V_i$. Let $(Z,\mathfrak Z_i)$ be a tuple of objects satisfying the following properties:
\begin{itemize}
\item $Z\subset X$ is a closed analytic subvariety of $X$,
\item $\mathfrak Z_i$ is a closed formal subsecheme of $\mathfrak V_i$,
\item the following identities hold
$$(\mathfrak Z_i)_\eta=\sp^{-1}(V_i)\cap Z, \forall i,$$
where $\sp\colon X\to \mathfrak X$ denotes the specialization map.
\end{itemize}
Then there is a closed formal subscheme $\mathfrak Z$ of $\mathfrak X$ such that $\mathfrak Z_i=\mathfrak Z\cap \mathfrak V_i$ for all $i\in I$ and
$$\mathfrak Z_\eta=Z.$$
 \end{theorem}
\begin{proof}
The theorem is obtained by applying \cite[Prop.7.5.2]{deJ} finitely many times.
\end{proof}
\begin{theorem}\label{modelization}
Let $X,Y,Z$ be smooth special rigid $K$-varieties and let $\mathfrak Y$ be a model of $Y$.
\begin{itemize} 
\item[(i)] If $\phi\colon Z\to Y$ is a closed immersion, then there exists a model $\mathfrak Z$ of $Z$ and a proper morphism $\varphi\colon \mathfrak Z\to \mathfrak Y$ such that $\varphi_\eta =\phi$.
\item[(ii)] If $\mathfrak X$ and $\mathfrak X'$ are two formal models of $X$, then there exist two morphisms $h: \mathfrak X''\to \mathfrak X$ and $h': \mathfrak X''\to \mathfrak X'$ of models of $X$ such that the induced morphisms ${h}_\eta$ and ${h'}_\eta$ are isomorphisms. 
\end{itemize} 
\end{theorem}

\begin{proof}
(i). We first prove the statement (i) for the special case when $\mathfrak Y=\Spf A$ with the largest ideal of definition $J$ of $A$ being generated by a regular system of elements in $A$. We will use the constructions in \cite[$\S$7.5]{deJ} and \cite{deJ1}.
	
	
	
Let $A[J^n/\varpi]$ be the subalgebra of $A\otimes_R K$ generated by $A$ and elements of form $i/\varpi$ where $i\in J^n$. Let $B_n$ be the $J$-adic completion of $A[J^n/\varpi]$ and $C_n=B_n\otimes_R K$. Note that $A[J^n/\varpi]$ is regular since $A$ is regular with a system of regular elements $x_1,\ldots,x_n,y_1,\ldots,y_m$ and $A[J^n/\varpi]/J=A/J$ is regular. Hence by \cite[IV$_1$, Lemme (17.3.8.1)]{GD60} $B_n$ is regular and so normal.
	
Let $V_n=\Sp\ C_n$.  Let $\alpha_n$ and $\alpha'_n$ be morphisms defined as
$$\alpha_n\colon B_n\to C_n\overset{\alpha'_n}{\rightarrow} \Gamma(Z\cap V_n,\mathcal{O}_X)$$
where $\alpha'_n$ is induced from the inclusion $Z\cap V_n\to V_n$. We denote by $I_n=\ker \alpha_n$ the kernel of $\alpha_n$ in $B_n$ and by $I'_n$ the kernel of $\alpha'_n$ in $C_n$. Since $C_n$ is an affinoid $K$-algebra, the Maximum Modulus Principle holds for the norm $|\ |_{\sup}$ on $C_n$ (\cite[6.2.1/4]{BGR84}). That is, for every $f\in C_n$, there exists $x\in \Sp C_n$ such that $|f(x)|=|f|_{\sup}$. This implies that for every $f\in C_n$, there exists $c\in K$ such that $|cf|_{\sup}\leq 1$. Therefore $I'_n$ can be generated by power-bounded elements (an elements $f\in  C_n$ is {\em power-bounded } if $|f|_{\sup}\leq 1$, \cite[6.2.3/1]{BGR84}). Applying the isomorphism in \cite[Theorem 7.4.1]{deJ} for the normal ring $B_n$ we can show that  $I'_n=I_{n}C_n$. Moreover, by the functoriality of $B_n$ and $C_n$  we have, for each $m>n$, the following commuative diagram
\begin{displaymath}
\xymatrix@=3 em{
B_{m}\ar[r]^{}\ar[d]_{}& C_{m}  \ar[d]^{}\ar[r]^{\alpha'_{m}\qquad}& \Gamma(Z\cap V_{m},\mathcal{O}_Z) \ar[d]^{}\\
B_{n}\ar[r]_{}& C_n\ar[r]_{\alpha'_{n}\qquad}& \Gamma(Z\cap V_n,\mathcal{O}_Z)
}
\end{displaymath}
which gives $I'_n=I'_{m}C_n$ by applying \cite[7.2.2/6]{BGR84} for the affinoid subdomain $\Sp C_n$ of $\Sp C_m$. Since $I'_n=I_{n}C_n$, we have $I_n=I_{m}B_n$.
It follows from \cite{deJ1} (see also \cite[$\S$7.1.13]{deJ}) that there exists an integer $c>0$ and, for each $n> c$, a surjective morphism
$$\beta_n\colon B_n\to A/J^{n-c}$$
satisfying the following compatibility condition: for each $m>n>c$, the diagram
\begin{displaymath}
\xymatrix@=3 em{
	B_{m} \ar[d]_{}\ar[r]^{\beta_{m}}& A/J^{m-c}\ar[d]^{}\\
	B_n\ar[r]_{\beta_{n}}& A/J^{n-c}
}
\end{displaymath}
commutes. Since $I_n=I_{m}B_n$, the morphisms
$$\beta_{m}(I_{m})\to \beta_n(I_n)$$ 
are surjective. Define the limit ideal
$$\mathfrak a:=\varprojlim \beta_n(I_n)\subset \varprojlim A/J^{n-c}=A.$$
We will show that $\mathfrak Z=\Spf A/\mathfrak a$, i.e.~ showing that $Z=\left(\Spf A/\mathfrak a\right)_\eta$. By definition, $Z\cap V_n=\left(\Spf\ B_n/I_n\right)_\eta$ and $\left(\Spf A/\mathfrak a\right)_\eta\cap V_n =\left(\Spf A/\mathfrak a\otimes_A B_n(A)\right)_\eta$ since the functor $\eta$ commutes with tensor product. Therefore it suffices to show that 
\begin{align*}
\Spf A/\mathfrak a\otimes_A B_n(A)\cong B_n/I_n.
\end{align*}
Let  $\varphi_n\colon A\to B_n$ be the natural morphism. We will show the following identity (via $\varphi_n$) 
\begin{align} \label{32}
I_n=\mathfrak a B_n.
\end{align}
Consider the following diagram
\begin{displaymath}
\xymatrix@=3 em{
	A  \ar[d]_{\varphi_m}\ar[dr]^{\gamma_m}&\\
	B_{m} \ar[d]_{\varphi^m_n}\ar[r]^{\beta_{m}}& A/J^{m-c}\ar[d]^{}\\
	B_n\ar[r]_{\beta_{n}}& A/J^{n-c}.
}
\end{displaymath}
We claim that
\begin{align}\label{33}
(\ker \beta_m) B_n\subset J^{m-n}B_n.
\end{align}
Indeed, let $x\in B_m$ and let
$$x=\sum_{i\geq 0}\frac{a_i}{\varpi^i}$$
with $a_i\in  J^{im}$ for all $i\geq 0$ be a representative of $x$. Then, the morphism $\beta_m$ can be expressed as 
$$\beta_m(x)=a_0\in A/J^{m-c}.$$
If $x\in \ker \beta_m$, then $a_0\in J^{m-c}B_n\subset J^{m-n}B_n$ and 
$$\frac{a_i}{\varpi^i}\in J^{i(m-n)}B_n\subset J^{m-n}B_n \text{ for all } i\geq 1,$$
which gives (\ref{33}). We first prove the inclusion $\mathfrak a B_n\subset I_n$ of (\ref{32}). Take $x\in \mathfrak a$ then $\beta_m\left(\varphi_m(x)\right)=\gamma_m(x)\in \beta_m(I_m)$ since $\mathfrak a=\varprojlim \beta_n(I_n)$. Therefore, $\beta_m\left(\varphi_m(x)\right)=\beta_m(x_m)$ for some $x_m\in I_m$, which yields that
$$\varphi_m(x)-x_m\in \ker \beta_m.$$
Then
$$\varphi_n(x)-\varphi^m_n(x_m)\in (\ker \beta_m)B_n,$$
which implies, by combining (\ref{33}) and the identity  $I_n=I_{m}B_n$, that
$$\varphi_n(x)\in I_n+J^{m-n}B_n,\ \forall m>n.$$
Therefore $\varphi_n(x)\in I_n$, since $\cap_{m>n}J^{m-n}B_n=0$. To show the other inclusion $I_n\subset \mathfrak a B_n$ of  (\ref{32}) we take $y\in I_n$. Then, since  $I_n=I_mB_n$,
$$y=\sum_{\nu} \varphi^m_n(x_{\nu}) y_{\nu} \text{ for some } x_{\nu}\in I_m  \text{ and } y_{\nu}\in B_n.$$
Notice that, $\mathfrak a=\varprojlim \beta_m(I_m)$ and the system $\beta_m(I_m)$ is surjective. It follows that, there exists, for each  $\nu$, elements $a_{\nu}\in \mathfrak a$ such that $\beta_m(x_{\nu})=\gamma_m(a_{\nu})=\beta_m(\varphi_m(a_\nu))$ and therefore $$x_\nu-\varphi_m(a_\nu)\in\ker \beta_m.$$
Using the inclusion (\ref{33}), we may deduce that
$$\varphi^m_n(x_{\nu}) \in \varphi_n(\mathfrak a)B_n+J^{m-n}B_n,\ \forall m>n.$$
So $y \in \varphi_n(\mathfrak a)B_n+J^{m-n}B_n,\ \forall m>n$, and therefore $y\in \varphi_n(\mathfrak a)B_n$ according to the identity $\cap_{m>n}J^{m-n}B_n=0$. This proves  (\ref{32}). The identity  (\ref{32}) induces a morphism $A/\mathfrak a\to B_n/I_n$ making the following diagram commutative

\begin{displaymath}
\xymatrix@=3 em{
	A \ar[d]_{}\ar[r]^{p}& A/\mathfrak a\ar[d]^{}\\
	B_n\ar[r]&B_n/I_n
}
\end{displaymath}
This  gives rise to a morphism $\phi\colon  B_n(A)\otimes_A A/\mathfrak a\to B_n/I_n$ and an induced commutative diagram:
\begin{displaymath}
\xymatrix@=3 em{
	A  \ar[d]_{\varphi_n}\ar[r]^{p}& A/\mathfrak a\ar[d]^{\bar{\varphi}_n}\ar[ddr]_{}&\\
	B_n(A)\ar[r]_{p_{n}}\ar[drr]_{}& B_n(A)\otimes_A A/\mathfrak a\ar[dr]^{\phi}&\\
&&B_n/I_n
}
\end{displaymath}
It is obvious that the morphism $\phi$ is surjective. To prove its injectivity, we  take $\bar y\in \ker \phi$ with a preimage $y\in I_n$ in $B_n$. It follows from the identity  (\ref{32}) that
$$y=\sum_{\nu}\varphi_n(a_\nu) \cdot y_\nu$$
for some $a_\nu \in \mathfrak a$ and $y_\nu\in B_n$. Then
$$\bar y=p_n(y)=\sum_{\nu}p_n\left(\varphi_n(a_\nu)\right)\cdot p_n(y_\nu)=\sum_{\nu}\bar{\varphi}_n\left(p(a_\nu)\right) \cdot p_n(y_\nu)=0,$$
and hence one obtains the injectivity of $\phi$. This proves that the morphism $$\phi\colon  B_n(A)\otimes_A A/\mathfrak a\to B_n/I_n$$ is an isomorphism. Hence
$$Z\cap V_n=\left(\Spf\ B_n/I_n\right)_\eta=\left(\Spf A/\mathfrak a\otimes_A B_n(A)\right)_\eta=\left(\Spf A/\mathfrak a\right)_\eta\cap V_n$$
for all $n\geq 1$, and therefore $Z=\left(\Spf A/\mathfrak a\right)_\eta$.

We now construct a closed immersion of formal schemes $\varphi\colon \mathfrak Z\to \mathfrak Y$  such that $\mathfrak Z_\eta=Z$ for general $\mathfrak Y$. Using resolution of singularities of $\mathfrak Y$, we may assume that $\mathfrak Y$ is regular and that $\mathfrak Y_s=\sum_{i\in S}N_i\mathfrak E_i$ is a strict normal crossings divisor.  We define, as in Definition \ref{Poincareseries}, for each subset $I\subset S$, $E^{\circ}_I:=\cap_{i\in I}E_i\setminus \cup_{j\not\in I}E_j$ and denote by $\mathfrak Y_I$ the formal completion of $\mathfrak Y$ along $E^{\circ}_I$. Note that $\mathfrak Y_I$ can be covered by an admissible covering of open affine formal schemes $\mathfrak U_i$ such that all $\mathfrak U_i$ and $\mathfrak U_i\cap \mathfrak U_j$ satisfy the assumption of the special case considered above. Moreover, since the above construction for the affine case is functorial we may therefore glue affine pieces to obtain a closed formal subscheme $\mathfrak Z_I$ of $\mathfrak Y_I$ such that  $(\mathfrak Z_I)_\eta=( \mathfrak Y_I)_\eta\cap Z$. 
Since
$$\mathfrak Y_0=\sqcup_{I\subset S} E^{\circ}_I$$
is a decomposition of $\mathfrak Y_0$ by locally closed subvarieties, it follows from Theorem \ref{descent} that the tuple $\left((\mathfrak Z_I)_{I\subset S}, Z\right)$ comes from a closed formal subscheme of $\mathfrak Y$. That is, there is a closed formal subscheme $\mathfrak Z$ of $\mathfrak Y$ such that
$$\mathfrak Z_\eta=Z \text{ and }\mathfrak Z_I=\mathfrak Z\cap \mathfrak Y_I,\ \forall I\subset S.$$
This completes (i).

(ii) Let us denote by $\mathfrak Y$ the fiber product $\mathfrak X\times_R \mathfrak X'$ and by $Y$ its generic fiber $\mathfrak Y_\eta$. By \cite[7.2.4]{deJ}, $X$ is separated and $Y\cong X\times X$, and therefore the diagonal morphism $\phi\colon X\to X\times X\cong Y$ is a closed immersion. It then follows from (i) that, there exist a formal model $\mathfrak X''$ and a morphism $\varphi\colon \mathfrak X''\to \mathfrak Y$ such that $\varphi_\eta=\phi$. Hence we obtain the morphisms $h=p_1\circ \varphi$ and $h'=p_2\circ \varphi$ as expected, where $p_1\colon \mathfrak X\times_R \mathfrak X'\to \mathfrak X$ and $p_2\colon \mathfrak X\times_R \mathfrak X'\to  \mathfrak X'$ are the canonical projections.
\end{proof}

Recall that, for $n\in \mathbb N^*$, we put already $R(n)=R[\tau]/(\tau^n-\varpi)$, $K(n)=K[\tau]/(\tau^n-\varpi)$, and for each formal $R$-scheme $\mathfrak X$ we denoted $\mathfrak X(n)=\mathfrak X\times_RR(n)$. Let $\theta_{\mathfrak X}$ be the action of $\mu_n=\Spec\big(k[\xi]/(\xi^n-1)\big)$ on $\mathfrak X(n)$ induced from the natural action of $\mu_n$ on $R(n)$. Let $X$ be a smooth special rigid $K$-variety and let $\omega$ is a gauge form on $X$. We denote $X(n)=X\times_K K(n)$ and by $\omega(n)$ the pullback of $\omega$ via the natural morphism $X(n)\to X$.

\begin{lemma-def}\label{int-Rigid-Xm}
Let $X$ be a smooth special rigid $K$-variety. Then, for any two special formal models $\mathfrak X$ and $\mathfrak X'$ of $X$, the pairs $(\mathfrak X(n),\theta_{\mathfrak X})$ and $(\mathfrak X'(n),\theta_{\mathfrak X'})$ are $\mu_n$-pairs of $X(n)$, and they are equivalent. We define {\bf the $\mu_n$-integral of $\omega(n)$ on $X(n)$} to be the $\mu_n$-integral of $\omega(n)$ on $X(n)$ with this action, i.e.~
$$\int_{X(n)}|\omega(n)|:=\int_{\mathfrak X_0}\int_{\mathfrak X(n)}|\omega(n)|\ \in \mathscr M_{k}^{\mu_n}.$$
\end{lemma-def}

\begin{proof}
The statement that $(\mathfrak X(n),\theta_{\mathfrak X})$ and $(\mathfrak X'(n),\theta_{\mathfrak X'})$ are $\mu_n$-pairs of $X(n)$ is trivial since $\mathfrak X(n)_{\eta}=\mathfrak X_{\eta}(n)$. By Theorem \ref{modelization}, there exist two morphisms $h \colon \mathfrak X''\to \mathfrak X$ and $h' \colon \mathfrak X''\to \mathfrak X'$ of models of $X$ such that ${h}_\eta$ and ${h'}_\eta$ are isomorphisms. It implies that $\mathfrak X''(n)_{\eta}\cong \mathfrak X(n)_{\eta}$ and $\mathfrak X''(n)_{\eta}\cong \mathfrak X'(n)_{\eta}$. By the naturality of the $\mu_n$-action on the $n$-ramification $R(n)$, the induced morphims $h(n) \colon \mathfrak X''(n)\to \mathfrak X(n)$ and $h'(n) \colon \mathfrak X''(n)\to \mathfrak X'(n)$ are $\mu_n$-equivariant. Hence, by definition, the $\mu_n$-pairs $(\mathfrak X(n),\theta_{\mathfrak X})$ and $(\mathfrak X'(n),\theta_{\mathfrak X'})$ of $X(n)$ are equivalent.	
\end{proof}

\begin{definition}
Let $X$ be a smooth special rigid $K$-variety, and let $\omega$ be a gauge form on $X$. The formal power series
\begin{align*}
P(X,\omega; T):=\sum_{n\geq 1}\Big(\int_{X(n)}|\omega(n)|\Big)T^n\ \text{ in }\ \mathscr M_{k}^{\hat \mu}[\![T]\!]
\end{align*}
is called {\it the volume Poincar\'e series of $(X,\omega)$}.
\end{definition}

\begin{prop-def}\label{MVidentities}
Let $X$ be a smooth special rigid $K$-variety, let $\mathfrak X$ be a special formal model of $X$. Then the quantity
$$\int_{\mathfrak X_0}\MV(\mathfrak X)\ \in \mathscr M_{k}^{\hat\mu}$$
is independent of the model $\mathfrak X$. We call it the {\bf motivic volume} of $X$ and denote by $\MV(X)$.
\end{prop-def}

\begin{proof}
{\it Step 1.} We asume that $X$ admits an $\mathfrak X$-bounded gauge form $\omega$. Using resolution of singularities we can assume that $\mathfrak X$ is regular with strict normal crossing divisor $\mathfrak X_s=\sum_{i\in S}N_i\mathfrak E_i$. From Definition \ref{defMVformal} we have
\begin{align}\label{eq1}
\int_{\mathfrak X_0}\MV(\mathfrak X)=\sum_{\emptyset\not=I\subseteq S}(1-\L)^{|I|-1}\big[\widetilde{E}_I^{\circ}\to \Spec k \big]\ \in \mathscr M_k^{\hat\mu}.
\end{align}
Using Theorem \ref{int_Xm} and Lemma-Definition \ref{int-Rigid-Xm} we have
$$
\int_{X(n)}|\omega(n)|=\L^{-d}\sum_{\emptyset\not=I\subseteq S}(\L-1)^{|I|-1}\big[\widetilde{E}_I^{\circ}\to \Spec k\big]\left(\sum_{\begin{smallmatrix} k_i\geq 1, i\in I\\ \sum_{i\in I}k_iN_i=n \end{smallmatrix}}\L^{-\sum_{i\in I}k_i\alpha_i}\right),
$$
with $\alpha_i=\ord_{\mathfrak E_i}(\omega)$, thus
\begin{align}\label{eq3}
P(X,\omega; T)=\L^{-d}\sum_{\emptyset\not=I\subseteq S}(\L-1)^{|I|-1}\big[\widetilde{E}_I^{\circ}\to \Spec k\big]\prod_{i\in I}\frac{\L^{-\alpha_i}T^{N_i}} {1-\L^{-\alpha_i}T^{N_i}},
\end{align}
from which, together with (\ref{eq1}),
\begin{align}\label{eq2}
\int_{\mathfrak X_0}\MV(\mathfrak X)=-\L^d \lim_{T\to \infty}P(X,\omega; T).
\end{align}
The equality (\ref{eq2}) guarantees that $\int_{\mathfrak X_0}\MV(\mathfrak X)$ is independent of the model $\mathfrak X$ of $X$.


{\it Step 2.} We do not assume that $X$ admits an $\mathfrak X$-bounded gauge form. Let $\mathfrak X'$ be another special formal models of $X$. Since $\mathfrak X$ (resp. $\mathfrak X'$) admits a resolution of singularities, we can identify $\mathfrak X$ (resp. $\mathfrak X'$) with a resolution of singularities for it. By \cite[Prop.-Def. 7.38]{Ni2}, the special formal $R$-scheme $\mathfrak X$ (resp. $\mathfrak X'$) has a finite open covering $\{\mathfrak U_i\}_{i\in Q}$ (resp. $\{\mathfrak U'_j\}_{j\in Q'}$) such that each $(\mathfrak U_i)_{\eta}$ (resp. $(\mathfrak U'_j)_{\eta}$) admits a $\mathfrak U_i$-bounded (resp. $\mathfrak U'_j$-bounded) gauge form $\omega_i$ (resp. $\omega'_j$). By Proposition \ref{MV-additive},
$$\MV(\mathfrak X)=\sum_{I\subseteq Q}(-1)^{|I|-1} \MV(\mathfrak U_{I})$$
and
$$\MV(\mathfrak X')=\sum_{J\subseteq Q'}(-1)^{|J|-1} \MV(\mathfrak U'_{J}).$$
By Theorem \ref{modelization}, there exist two morphisms $\mathfrak h \colon \mathfrak X''\to \mathfrak X$ and $\mathfrak h' \colon \mathfrak X''\to \mathfrak X'$ of models of $X$ such that $\mathfrak{h}_\eta$ and $\mathfrak h'_\eta$ are isomorphisms. Put $\mathfrak V_i=\mathfrak h^{-1}(\mathfrak U_i)$ and $\mathfrak V'_j=\mathfrak h^{\prime-1}(\mathfrak U'_j)$ for every $i\in Q$ and $j\in Q'$. By Proposition \ref{changeofvariables} and Lemma-Definition \ref{int-Rigid-Xm}, we get (with some $i\in I$ and $j\in J$):
$$\int_{(\mathfrak V_I)_{\eta}(n)}|\mathfrak h^*\omega_i(n)|=\int_{(\mathfrak U_I)_{\eta}(n)}|\omega_i(n)|$$
and
$$\int_{(\mathfrak V'_J)_{\eta}(n)}|\mathfrak h^{\prime *}\omega'_j(n)|=\int_{(\mathfrak U'_J)_{\eta}(n)}|\omega'_j(n)|,$$
for every $n\in \mathbb N^*$. These equalities together with the computation in Step 1 give us
$$\int_{(\mathfrak V_I)_0}\MV(\mathfrak V_{I})=\int_{(\mathfrak U_I)_0}\MV(\mathfrak U_{I})$$
and
$$\int_{(\mathfrak V'_J)_0}\MV(\mathfrak V'_{J})=\int_{(\mathfrak U'_J)_0}\MV(\mathfrak U'_{J}).$$
Hence, by Proposition \ref{MV-additive}, as well as some above equalities, we have
$$\int_{\mathfrak X_0}\MV(\mathfrak X)=\int_{\mathfrak X''_0}\MV(\mathfrak X'')=\int_{\mathfrak X'_0}\MV(\mathfrak X').$$
This means that $\int_{\mathfrak X_0}\MV(\mathfrak X)$ is independent of the model $\mathfrak X$. 
\end{proof}
The following is a direct consequence of the proposition.
\begin{corollary}\label{cor510}
Let $X$ be a bounded smooth special rigid $K$-variety which admits a $\mathfrak X$-bounded gauge. Let $h\colon \mathfrak Y \to \mathfrak X$ be a resolution of singularities of $\mathfrak X$. With the notation as in Definition \ref{resolutionofsingularities} we have
$$\MV(X)=\sum_{\emptyset\not=I\subseteq S}(1-\L)^{|I|-1}\big[\widetilde{E}_I^{\circ}\big]\ \in \mathscr M_k^{\hat\mu}.$$
\end{corollary}

\begin{definition}[Bounded analytic spaces]\label{boundedvar}
Let $X$ be a smooth special rigid $K$-variety. A differential form $\omega$ on $X$ is called {\it bounded} if it is $\mathfrak X$-bounded for some formal model $\mathfrak X$ of $X$. The rigid $K$-variety $X$ is called {\it bounded } if it admits a bounded gauge form.
\end{definition}
Notice that our definition of bounded rigid varieties is not related to that of \cite{NS1}. 
\begin{proposition}
Let $X$ be a bounded smooth special rigid $K$-variety and let $Y$ be a smooth closed subvariety of $X$. Then $Y$ is also bounded. In particular, all affine smooth special rigid varieties are bounded.
\end{proposition}
\begin{proof}
It follows from Theorem \ref{modelization}(i) and \cite[Proposition 2.6]{Ni2}.
\end{proof}

\begin{corollary}\label{cor511}
Let $X$ be a bounded smooth special rigid $K$-variety which admits a bounded gauge form $\omega$. Then the volume Poincar\'e series $P(X,\omega; T)$ is rational and the following identity holds in $\mathscr M_{k}^{\hat\mu}$:
$$\MV (X)=-\L^d \lim_{T\to\infty} P(X,\omega; T).$$
\end{corollary}

We believe that Corollary \ref{cor510} is still true when we do not assume the $\mathfrak X$-boundedness of a gauge form on $X$ (see the below conjecture). 

\begin{conjecture}
Let $X$ be a smooth special rigid $K$-variety, and $\omega$ a gauge form on $X$. Then the volume Poincar\'e series $P(X,\omega; T)$ is rational and the identity
$$\MV (X)=-\L^d \lim_{T\to\infty} P(X,\omega; T)$$
holds in $\mathscr M_{k}^{\hat\mu}$.
\end{conjecture}
\begin{example}\label{example31} 
Let $R=k[\![\varpi]\!]$, $K=k(\!(\varpi)\!)$, $R(n)=k[\![\varpi^{1/n}]\!]$ and $K(n)=k(\!(\varpi^{1/n})\!)$ for $n\in \mathbb N^*$. Consider the formal schemes $\mathfrak X=\Spf(R[\![x_1,\dots,x_d]\!])$, $\mathfrak Y=\Spf(R\{x_1,\dots,x_d\})$ and  $\mathfrak U=\Spf(R\{\frac{x} {\varpi^p}\})$  with $R\{\frac{x} {\varpi^p}\}:=R\{x,y\}/(x-\varpi^py)$ for new variables $y=(y_1,\ldots,y_{d})$. 
Then $\mathfrak X_\eta=D^d$ is the $d$-dimensional ``open" ball of radius $1$, and $\mathfrak Y_\eta=\mathbf{B}^d$ and $\mathfrak U_\eta=\mathbf{B}^d(|\varpi^p|)$ are the $d$-dimensional ``closed" balls of radius $1$ and $|\varpi^p|$ respectively. It follows from Example \ref{example22} that
$$\MV(D^d)=1,\ \MV(\mathbf{B}^d)=\MV\left(\mathbf{B}^d(|\varpi^p|)\right)=\L^d\in \mathscr M_{k}^{\hat\mu}.$$
\end{example}


\subsection{Motivic volume morphism}
Let $\mathrm{SSRig}_K$ denote the category of special smooth rigid $K$-varieties. In this subsection, we define a (motivic volume) homomorphism from a certain Grothendieck ring of $\mathrm{SSRig}_K$ to the ring $\mathscr M_{k}^{\hat\mu}$, which is compatible with the motivic volume defined in the previous section. 

We first introduce the notion of special rational subdomains. Let $X$ be an affine special rigid $K$-variety with a coordinate ring $A$. Assume that $g,f_1,\ldots,f_n\in A$ generate the unit ideal in $ A{\otimes}_R K$. Writing $\frac{f}{g}$ for $\big(\frac{f_1}{g},\ldots,\frac{f_n}{g}\big)$ we define
\begin{align*}\label{ratdomain}
X\Big(\frac{f}{g}\Big):=\left\{x\in X \mid |f_i(x)|\leq |g(x)|,\ \forall i\right\}.
\end{align*}
A subvariety of $X$ in this form is called {\it a special rational subdomain of $X$}. In general, a subvariety $Y$ of a special rigid $K$-variety $X$ is called {\it special rational subdomain of $X$} if there exists a finite cover by affine special varieties $(X_i)_{i\in I}$ of $X$ such that $Y\cap X_i$ is a special rational subdomain of $X_i$ for every $i\in I$. Notice that a rational subdomain $Y$ of $X$ and its complement are aslo objects of $\mathrm{SSRig}_K$, since for instance, if $X$ is affine, then $Y$ and $X\setminus Y$ have coordinate rings $$A\{z_1,\ldots,z_n\}/(f_1-z_1g,\ldots,f_n-z_ng)$$ and 
$$A]\!]z]\!]/(zf_1-g)\cdots (zf_n-g),$$
respectively. Define $\mathbf{K}(\mathrm{SSRig}_K)$ the abelian group generated by the isomorphism classes $[X]$ of $\mathrm{SSRig}_K$ modulo the relation 
$$[X]=[Y]+[X\setminus Y]$$
where $Y\subseteq X$ is a special rational subdomain of $X$. The group $\mathbf{K}(\mathrm{SSRig}_K)$ admits a ring structure whose multiplication is induced by fiber product.

\begin{theorem}\label{MV-morphism}
There exists a unique ring homomorphism
$$\MV\colon\mathbf{K}(\mathrm{SSRig}_K)\to \mathscr M_k^{\hat{\mu}}$$
such that
$$\MV([X])=\MV(X)$$ 
for all objects $X$ of $\mathrm{SSRig}_K$.
\end{theorem}

\begin{proof}
It is obvious that $\MV(1)=1$. We first prove the additivity of $\MV$, i.e.  
$$\MV([X])=\MV([Z])+\MV([X\setminus Z])$$ 
for any rational subdomain $Z$ of $X$. In fact, by Proposition \ref{MV-additive} we can asume that $X$ is special affine, and let $A$ be a coordinate ring of $X$. By induction, we can assume that $Z=X(\frac{f}{g})$ for some $f,g\in A$ generating the unit ideal in $ A{\otimes}_R K$. We may assume further that $f$ and $g$ have no comon factors in $A$. Then $A\{z\}/(f-zg)$ and $A]\!]z]\!]/(zf-g)$ are coordinate rings of $Z$ and $X\setminus Z$, respectively. We consider the admissible blow up of $\mathfrak X:=\Spf A$ along the ideal generated by $f$ and $g$, say, $\mathfrak Y\to \mathfrak X$. Let us consider the following special $R$-algebras
$$A_1:=A\{z_1\}/(z_1f-g)\quad \text{and}\quad A_2:=A\{z_2\}/(f-z_2g).$$ 
As described in \cite[Lem. 2.18]{Ni2}, $\left\{\Spf A_1,\Spf A_2\right\}$ is an open cover of $\mathfrak Y$, and the intersection $\Spf A_1\cap\Spf A_2$ is isomorphic to 
$$\Spf A\{z_1,z_2\}/(z_1f-g,f-z_2g)\cong\Spf A_1\{z_2\}/(1-z_1z_2).$$
We first observe that $A_2$ is a coordinate ring of $Z$, and the intersection $\Spf A_1\cap\Spf A_2$ can  be identified with the formal completion of $\Spf A_1$ along the open subset of the special fiber $(\Spf A_1)_0$ defined by $z_1\neq 0$. On the other hand, the formal completion of $\Spf A_1$ along the closed subset defined by $z_1= 0$ is a $A]\!]z_1]\!]/(z_1f-g)$, a coordinate ring of $X\setminus Z$. It then  follows from Proposition \ref{MV-additive}  that
$$\MV(\Spf A_1)=\MV(\Spf A_1\cap\Spf A_2)+\MV\left(\Spf \left(A]\!]z_1]\!]/(z_1f-g)\right)\right)$$
Hence, by the additivity of the motivic volume of formal $R$-schemes (cf. Proposition \ref{MV-additive}) the following identities
\begin{align*}
\MV(\mathfrak Y)&=\MV(\Spf A_1)+\MV(\Spf A_2)-\MV(\Spf A_1\cap\Spf A_2)\\
&=\MV(\Spf A_2)+\MV\left(\Spf \left(A]\!]z_1]\!]/(z_1f-g)\right)\right),
\end{align*}
hold in $\mathscr M_{\mathfrak Y_0}^{\hat{\mu}}$. Applying the push-forward morphism 
$$\int_{\mathfrak Y_0}\colon \mathscr M_{\mathfrak Y_0}^{\hat{\mu}}\to \mathscr M_k^{\hat{\mu}}$$
we obtain the identity $\MV(X)=\MV(Z)+\MV(X\setminus Z)$ in $\mathscr M_k^{\hat{\mu}}$.
	
We are going to prove the multiplicativity of $\MV$, more precisely
$$\MV (X\times Y) =\MV (X)\cdot \MV (Y)$$
where $X$ and $Y$ are smooth special rigid $K$-varieties.
Indeed, we first prove this identity for bounded rigid $K$-varieties. Assume that  $X$ and $Y$ are bounded (cf. Definition \ref{boundedvar}) of dimension $d_1$ and $d_2$ with bounded gauge form $\omega_X$ and $\omega_Y$, respectively. Then $\omega:=\omega_X\otimes\omega_Y$ is a bounded gauge form on $X\times Y$. Let $P(X,\omega_X; T)$, $P(Y,\omega_Y; T)$ and $P(X\times Y,\omega; T)$ be the volume Poincar\'e series of the pairs $(X,\omega_X)$, $(Y,\omega_Y)$ and $(X\times Y,\omega)$, respectively. By Lemma \ref{productofint}, for all $n\geq 1$ we have
$$\int_{(X\times Y)(n)}|\omega(n)|=\int_{X(n)}|\omega_X(n)|\cdot \int_{Y(n)}|\omega_Y(n)|$$
and therefore, 
$$P(X\times Y,\omega; T)=P(X,\omega_X; T)\ast P(X,\omega_X; T).$$ 
where $\ast$ is the Hadamard product of formal power series in $\mathscr M_k^{\hat\mu}$. By Corollary \ref{cor510} we have
\begin{align*}
\MV(X\times Y)&=-\L^{d} \lim_{T\to\infty} P(X\times Y,\omega; T)\\
&=-\L^{d} \lim_{T\to\infty} \left(P(X,\omega_X; T)\ast P(Y,\omega_Y; T)\right)\\
&=\Big(-\L^{d_1} \lim_{T\to\infty} P(X,\omega_X; T)\Big)\cdot  \Big(-\L^{d_2} \lim_{T\to\infty}P(Y,\omega_Y; T)\Big)\\
&=\MV (X)\cdot \MV (Y),
\end{align*}
where $d=d_1+d_2$ and the third identity is due to Lemma \ref{Lem2}.
	
Now, we are able to prove the multiplicativity of $\MV$ for general rigid $K$-varieties. Using resolution of singularities and \cite[Cor. 7.27]{Ni2} we show that $X$ can be covered by open bounded rigid varieties $(X_i)_{i\in I}$ such that $X_{I'}:=\cap_{i\in I'} X_i$ are also bounded for all $I'\subseteq I$, and that
$$\MV(X)=\sum_{I'\subseteq I}(-1)^{|I'|-1}\MV(X_{I'}).$$
We also take such a cover $(Y_j)_{j\in J}$ for $Y$, then obtain a cover $(X_i\times Y_j)_{(i,j)\in I\times J}$ for $X\times Y$. Applying the bounded case we get
$$\MV (X_{I'}\times Y_{J'}) =\MV (X_{I'})\cdot \MV (Y_{J'})$$
for every $I'\subseteq I$ and $J'\subseteq J$, and hence 
$$\MV (X\times Y) =\MV (X)\cdot \MV (Y)$$
according to the additivity of $\MV$.
\end{proof} 

\subsection{A motivic Fubini theorem}
We first slightly modify the notion of costructible functions and their integrals in \cite{NP}. Recall that a subset of $\R^n$ is {\em semi-algebraic} if it is a finite union of sets of forms 
$$\left\{x=(x_1,\ldots,x_n)\in \R^n\mid P(x_1,\ldots,x_n)=0; Q(x_1,\ldots,x_n)>0,\right\}$$
where $P(x)$ and $Q(x)$ are systems of polynomials. 


Let $A$ be an abelian group and let $V$ be a semi-algebraic set. A function $\varphi\colon V\to A$ is called {\em constructible} if there exists  a partition of $V$ into finitely many semi-algebraic subsets $\rho_1,\ldots,\rho_k$ such that $\varphi$ takes a constant value $a_i\in A$ on $\rho_i$ for each $i$. A constructible function $\varphi: V \rightarrow A$  can be written as a finite sum
$$
\varphi=\sum_{i =1}^k a_i 1_{\rho_i}
$$
where $1_{\rho_i}$ is the characteristic function of $\rho_i$. If $\varphi$ is a constructible function, the {\em Euler integral} of $\varphi$ is defined as
$$
\int_V \varphi d \chi_c=\sum_{i =1}^k  a_i \chi_c\left(\rho_i\right) .
$$
Let $f: V \rightarrow W$ be a continuous semi-algebraic map and let $\varphi: V \rightarrow A$ be a constructible function. The push forward $f_* \varphi$ of $\varphi$ along $f$ is the function 
$f_* \varphi: W \rightarrow A$ defined by
$$
f_* \varphi(y)=\int_{f^{-1}(y)} \varphi d \chi_c .
$$
\begin{theorem}[Change of variables formula]\label{change} Let $f: V \rightarrow W$ be a continuous semi-algebraic map and let $\varphi$ be a constructible function on $V$. Then, $f_* \varphi$ is also constructible and we have
$$
\int_W f_* \varphi d \chi_c=\int_V \varphi d \chi_c.
$$
\end{theorem}
\begin{proof}
See Statement 3.A in \cite{Viro88}.
\end{proof}
\begin{definition}{\rm Let $\Gamma$ be subset of $\Q^n$. $\Gamma$ is called a {\em polyhedron} if it is defined by $Ax^T\geq b^T$ for some $A\in \mathrm{Mat}(m,n;\Q)$ and $b\in \mathrm{Mat}(1,m;\Q)$. Extending this concept, {\em a constructible subset} $\Gamma$ of $\Q^n$ is a finite Boolean combination of polyhedra. Let $\Gamma_{\mathbb{R}}$ denote the canonical subset of $\mathbb{R}^n$ associated with $\Gamma$, defined by the same system of $\mathbb{Q}$-linear inequalities as $\Gamma$. A function $\varphi\colon \Gamma\to A$ is called {\em constructible} if there exists  a partition of $\Gamma$ into finitely many constructible subsets $\rho$ such that $\varphi$ takes a constant value $a_\rho \in A$ on each stratum $\rho$. Then the {\em Euler integral} of $\varphi$ is defined as $$\int_\Gamma \varphi d \chi_c:=\sum_{\rho} a_\rho \chi_c\left(\rho_\R\right) .$$}
\end{definition}

 Let $V\subset \R^n$ and $W\subset \R^m$ be constructible subsets. A map $f: V \rightarrow W$ is called a piecewise affine linear map if there is a partition of $V$ into finitely many constructible subsets $V_i$ such that $f|_{V_i}$ is a restriction of an affine linear map on $V_i$. Let $\varphi: V \rightarrow A$ be a constructible function. The {\em push forward} $f_* \varphi$ of $\varphi$ along $f$ is the function 
$f_* \varphi: W \rightarrow A$ defined by
$$
f_* \varphi(y)=\int_{f^{-1}(y)} \varphi d \chi_c .
$$
\begin{theorem}[Change of variables formula]\label{Qchange} Let $V\subset \Q^n$ and $W\subset \Q^m$ be constructible subsets. Let $f: V \rightarrow W$ be a piecewise affine linear map and let $\varphi\colon V\to A$ be a constructible function. Then, $f_* \varphi$ is also constructible and we have
$$
\int_W f_* \varphi d \chi_c=\int_V \varphi d \chi_c.
$$
\end{theorem}
\begin{proof}
By the additivity of the Euler integral, we may assume that $f$ is an affine linear map and that $\varphi$ is constant of value $a$ on $V$. Then
$$
f_* \varphi(y)=\int_{f^{-1}(y)} \varphi d \chi_c= a\chi_c\left(f^{-1}(y)\right).
$$
On the other hand, applying the change of variables formula (Theorem \ref{change}) for the map $f_\R\colon V_\R \rightarrow W_\R$ and the constructible function $a1_{V_\R}$ we have
$$\int_V \varphi d \chi_c=a\chi_c\left(V_\R\right)=\int_{V_\R} a1_{V_\R} d \chi_c=\int_{W_\R}  f_*( a1_{V_\R})d \chi_c=\int_W f_* \varphi d \chi_c.$$
\end{proof}
\begin{definition}{\rm 
The {\em bounded Euler characteristi}c $\chi^{\prime}(\Gamma)$ on the class of constructible subsets $\Gamma$ in $\Q^n$ is defined as follows. Let $\Gamma_{\mathbb{R}}$ be the canonical subset  of $\R^n$ associated with $\Gamma$, defined by the same system of $\mathbb{Q}$-linear inequalities as $\Gamma$. The compactly supported Euler characteristic of $\Gamma_{\mathbb{R}} \cap[-r, r]^n$ stabilizes for sufficiently large $r \in \mathbb{R}$. We define
$$\chi^{\prime}(\Gamma):=\chi_c\left(\Gamma_{\mathbb{R}} \cap[-r, r]^n\right)$$
for $r\gg 0$. In particular, if $\Gamma_\R$ is bounded in $\R^n$, then  $\chi^{\prime}(\Gamma)=\chi_{c}(\Gamma_\R)$. The bounded Euler characteristic is additive on disjoint unions and assigns the value 1 to every non-empty polyhedron, and remains invariant under affine linear automorphisms of $\Q^n$. If $\varphi\colon \Gamma\to A$ is constructible, then the {\em bounded Euler integral} of $\varphi$ is defined as $$\int_\Gamma \varphi d \chi^{\prime}:=\sum_{i \in I} a_i \chi^{\prime}\left(\rho_i\right) .$$
}\end{definition}
In the following we denote
$$\R_{>0}:=\left\{x\in \R \mid x>0\right\}$$
and similarly for $\R_{\geq 0}$ and $\Q_{\geq 0}$.
\begin{theorem}[Motivic Fubini Theorem]\label{fubini}
Let $X$ be a smooth special rigid $K$-varitety with a model $\mathfrak X$. Let $g=\{g_1,\ldots,g_r\}$ be a system of elements of $\Gamma(\mathfrak X,\mathcal{O}_{\mathfrak X})$. For each $\gamma\in  \Q^r_{\geq 0}$ we define the variety
$$X_{\gamma}:=\left \{x\in X\mid |g_i(x)| = |\varpi|^{\gamma_i} \right\}.$$
Then the function $\varphi_g\colon \Bbb Q^r_{\geq 0}\to \mathscr M_k^{\hat{\mu}}$ defined as
$$\varphi_g(\gamma)=\MV(X_{\gamma})$$
is constructible, and moreover, 
$$\MV(X)=\int_{ \Bbb Q^r_{\geq 0}} \varphi_g d\chi_c=\int_{ \Bbb Q^r_{\geq 0}}\varphi_g d \chi^{\prime}.$$
\end{theorem}
\begin{proof}
By additivity of the morphism $\MV$ we may assume that $X$ admits an $\mathfrak X$-bounded gauge form $\omega$. Let us consider a 
resolution of singularities $h\colon \mathfrak Y\to \mathfrak X$ of the formal  $R$-scheme $\mathfrak X$. Let $\mathfrak E_i$, $i\in S$, be the irreducible components of $(\mathfrak Y_s)_{\mathrm{red}}$.  Let $E_i:={(\mathfrak E_i)}_0$, $ E_I^{\circ}:=\cap_{i\in I}E_i\setminus \cup_{j\not\in I}E_j$ and let  $\widetilde{E}_I^{\circ}\to E_I^{\circ}$ be the covering with Galois group $\mu_{N_I}$ defined locally over $\mathfrak U_0\cap E_I^{\circ}$ as in Definition \ref{resolutionofsingularities}. Then for each $y\in {E}_I^{\circ}$ there exists an affine neighbourhood $\mathfrak U_I$ such that the following identity holds in $\Gamma(\mathfrak U_I,\mathcal{O}_{\mathfrak U_I})$
\begin{align*}
\tilde{\mathfrak{f}}:=h^{*}\mathfrak{f}&=u\prod_{i\in I} y^{N_i}_i
\end{align*}
 where $\mathfrak{f}$ denotes the structural morphism of $\mathfrak X$, $y_i$ is a local equation of $E_i$ at $y$ and $u$ is invertible in $\Gamma(\mathfrak U_I,\mathcal{O}_{\mathfrak U_I})$. Moreover, we can choose a resolution of singularities of such that the following identities hold in $\Gamma(\mathfrak U_I,\mathcal{O}_{\mathfrak U_I})$
\begin{align*}
\tilde{g_l}:=h^{*}{g_j}&=u_l\prod_{i\in I} y^{M_{il}}_i,\ \forall l=1,\ldots,r,
\end{align*}
where $u_l$ are invertible in $\Gamma(\mathfrak U_I,\mathcal{O}_{\mathfrak U_I})$.  Note that $N_i>0$ for all $i\in S$ while $M_{il}$ could be zero for some $i$ and $l$. 
Then one has
\begin{lemma}\label{Poicarecoeff}
Let $\alpha_i = \mathrm{ord}_{\mathfrak E_i}\omega$ for each $i\in I$. Then the following identities hold in $\mathscr M_{k}^{\hat{\mu}}$:
\begin{align*}
 \int_{X_{\gamma}(n)}|\omega(n)|=\L^{-d}\sum_{ \emptyset\neq I\subseteq S}(\L-1)^{|I|-1}\big[\widetilde{E}_I^{\circ}\big] \left(\sum_{\begin{smallmatrix} k_i\geq 1, i\in I\\ \sum_{i\in I}k_iN_i=n\\ \sum_{i\in I}k_i M_{ij}= \gamma_j n\end{smallmatrix}}\L^{-\sum_{i\in I}k_i\alpha_i}\right).
\end{align*}
Therefore,
\begin{align*}
\MV(X_{\gamma})&=-\sum_{\begin{smallmatrix}   I\subseteq S\\  I(\mathfrak{f})\neq \emptyset\end{smallmatrix}}(\L-1)^{|I(\mathfrak{f})|-1}\chi_c(\Delta_{I,\gamma})\big[\widetilde{E}_I^{\circ}\big],
\end{align*}
where $\Delta_{I,\gamma}$ is the cone in $\R^I_{>0}$ defined by the system of equations $\sum_{i\in I}(  M_{il}-\gamma_l N_i)k_i=0$. 
\end{lemma}
Before proving the lemma, we introduce the notion of $q$-separating resolution of singularities for some positive rational number $q$. Let $\mathfrak h\colon \mathfrak Y \to \mathfrak X$ be a resolution of singularities of $\mathfrak X$. With the notation as in Section \ref{motvol}, the resolution $\mathfrak h$ is called {\em  $q$-separating for a given element $f\in \Gamma(\mathfrak X,\mathcal{O}_{\mathfrak X})$} if for all $i\neq j\in S$, the condition $\mathfrak E_i\cap \mathfrak E_j\neq \emptyset$ implies $N_i+N_j>q$. Comparision with the notion of non-linearity  in \cite[Definition 4.1]{NS},  if $\mathfrak h$ is $q$-separating, then $n$ is not $\mathfrak Y_s$-linear for all $n\leq q$.
\begin{proof}[Proof of the lemma]
By \cite[Lemma 5.2]{NS} the quantities in the right hand sides of the lemma do not change if one blows up $\mathfrak Y$ along any intersection $E_I$. On the other hand, applying the argument in the proof of \cite[Lemma 2.9]{BdBLN}, we may construct a resolution of singularities by bowing up centers of form $E_I$ which is $n$-separating for $\mathfrak{f}$. That is,  if $\mathfrak E_i\cap \mathfrak E_j\neq \emptyset$ then $N_i+N_j>n$.  Therefore, it suffices to prove the following identity  holds in $\mathscr M_{k}^{\hat{\mu}}$
\begin{align*}
\int_{X_{\gamma}(n)}|\omega(n)|=\L^{-d}\sum_{\begin{smallmatrix}   N_i\mid n\\ M_{il}= \gamma_l N_i \end{smallmatrix}}\big[\widetilde{E}_i^{\circ}\big]\L^{- n\alpha_i/N_i}.
\end{align*}
Let $\mathfrak Z:=Sm\left(\widetilde{\mathfrak Y(n)}\right)$ the smooth locus of the normalization of $\mathfrak Y(n)$. It follows from \cite[Theorem 5.1]{Ni2} that
$$\mathfrak Z\to \mathfrak Y(n)$$
is a N\'eron smoothening (i.e. . ~ $\mathfrak Z$ adic smooth over $R(n)$, and the induced morphism $\mathfrak Z_{\eta} \to\mathfrak Y(n)_{\eta}$ is an open embedding satisfying $\mathfrak Z_{\eta}(K^{sh})=\mathfrak Y(n)_{\eta}(K^{sh})$) and
$$\mathfrak Z_0=\bigsqcup_{ N_i\mid n}  \widetilde{E}_i^{\circ}.$$
Let $S_{\gamma}:=\{i\in S| M_{il}\gamma=N_i, N_i\mid n\}$ and let 
$\mathfrak Z_{\gamma}$ the formal completion of $\mathfrak Z$ along $\bigsqcup_{ i\in S_{\gamma}}  \widetilde{E}_i^{\circ}$. Take a point $x$ of $ \mathfrak Z_\eta$, then by \cite[7.1.10]{deJ} (see Section \ref{section23}) $x$ corresponds to the equivalence class of $\varphi\colon\Spf R'\to  \mathfrak Z$ for some finite extension $R'$ of $R(n)$. Then its origin $\varphi_0\colon\Spec k'\to  \mathfrak Z_0$ belongs to $\mathfrak Z_0=\bigsqcup_{ i\in S_{\gamma}}  \widetilde{E}_i^{\circ}$, where $k'$ is the residue field of $R'$. Then $\varphi_0\in \widetilde{E}_i^{\circ}$ for some $i\in S, N_i\mid n$. Let $k_i:=\ord_\varphi y_i$ the order of $y_i$ at $\varphi$. Then 
$$\ord_\varphi(g_l)=M_{il}k_i \text{ and } \ord_\varphi(\mathfrak{f})=N_ik_i=n.$$
This means that
$$|g_l(x)|=|\varpi(n)|^{M_{il}k_i} \text{ and } |\mathfrak{f}(x)|=|\varpi(n)|^{N_ik_i}=|\varpi(n)|^n.$$
Therefore,  it is easily verified that $\varphi_0\in \bigsqcup_{ i\in S_{\gamma}}  \widetilde{E}_i^{\circ}$ if and only if $\varphi\in X_{\gamma}(n)$. Hence
$$\left(\mathfrak Z_{ \gamma}\right)_\eta =X_{\gamma}(n)\cap \mathfrak Z_\eta.$$
Moreover, it follows from \cite[Theorem 5.1]{Ni2} that $\mathfrak Z_\eta(K(n)^{sh})=X(n)(K(n)^{sh})$ and therefore 
$$\left(\mathfrak Z_{\gamma}\right)_\eta (K(n)^{sh})=X_{\gamma}(n)(K(n)^{sh}).$$
We deduce from Lemma \ref{openimmersion} that
$$\int_{X_{\gamma}(n)}|\omega(n)|=\int_{\left(\mathfrak Z_{\gamma}\right)_\eta}|\omega(n)|=\L^{-d}\sum_{\begin{smallmatrix}   N_i\mid n\\M_{il}=\gamma_l N_i \end{smallmatrix}}\big[\widetilde{E}_i^{\circ}\big]\L^{- n\alpha_i/N_i}.$$
This completes the first statement of the lemma. To prove the second statement we consider the Poincar\'e series
\begin{align*}
P(X_\gamma,\omega;T)&=\sum_{ n\geq 1}\int_{X_{\gamma}(n)}|\omega(n)|\ T^n\\
&=\L^{-d}\sum_{\emptyset\neq I\subseteq S}(\L-1)^{|I|-1}\big[\widetilde{E}_I^{\circ}\big] \sum_{n\geq 1}\left(\sum_{\begin{smallmatrix} k_i\geq 1, i\in I\\ \sum_{i\in I}k_iN_i=n\\ \sum_{i\in I}k_i M_{ij}= \gamma_j n\end{smallmatrix}}\L^{-\sum_{i\in I}k_i\alpha_i}\right)T^n\\
&=\L^{-d}\sum_{\emptyset\neq  I\subseteq S}(\L-1)^{|I|-1}\big[\widetilde{E}_I^{\circ}\big] \sum_{k\in \Delta_{I,\gamma}\cap \Bbb N^I} \L^{-\sum_{i\in I}k_i\alpha_i}T^{\sum_{i\in I}k_iN_i}
\end{align*}
Hence, it follows from Corollary \ref{cor510} and  (\ref{eq3.1}) that
\begin{align*}
\MV(X_\gamma)&=-\L^{d}\lim_{T\to\infty}P(X_\gamma,\omega;T)\\
&=-\sum_{\emptyset\neq I\subseteq S}(\L-1)^{|I|-1}\big[\widetilde{E}_I^{\circ}\big] \lim_{T\to \infty}\sum_{k\in \Delta_{I,\gamma}\cap \Bbb N^I} \L^{-\sum_{i\in I}k_i\alpha_i}T^{\sum_{i\in I}k_iN_i}\\
&=-\sum_{\emptyset\neq I\subseteq S}(\L-1)^{|I|-1}\big[\widetilde{E}_I^{\circ}\big]\chi_c( \Delta_{I,\gamma}).
\end{align*}
\end{proof}
We now are able to demonstrate that the function $\varphi_g\colon \Bbb Q^r_{\geq 0}\to \mathscr M_k^{\hat{\mu}}$, 
$$\varphi_g(\gamma)=\MV(X_{\gamma})=-\sum_{\begin{smallmatrix}   I\subseteq S\\  I(\mathfrak{f})\neq \emptyset\end{smallmatrix}}(\L-1)^{|I(\mathfrak{f})|-1}\big[\widetilde{E}_I^{\circ}\big]\chi_c( \Delta_{I,\gamma})$$ is constructible. Let us denote by $\ell_i(\mathbf{k})$ and $\ell(\mathbf{k})$ the linear maps $\sum_j M_{ij}k_j$ and $\sum_j N_jk_j$ respectively.
We also define the function $\Delta_I\colon  \R^I_{>0}\to \Bbb R^r_{\geq 0}$,
$$\Delta_I\left(\mathbf{k}\right):= \left(\frac{\ell_1(\mathbf{k})}{\ell(\mathbf{k})},\ldots,\frac{\ell_r(\mathbf{k})}{\ell(\mathbf{k})}\right).$$
Note that, the map $\Delta_I$ can be factorized as $\pi\circ L$ where $L\colon \R^I\to \R^{r+1}$ is the linear map defined as
$$L(\mathbf{k})=\left(\ell_1(\mathbf{k}),\ldots,\ell_r(\mathbf{k}),\ell(\mathbf{k})\right),$$
and $\pi\colon  \R^{r}\times\R_{>0}\to  \R^{r}, (x_1,\ldots,x_{r},z)\mapsto (\frac{x_1}{z},\ldots,\frac{x_r}{z})$.
Applying Minkowski-Weyl theorem we can deduce that  the image $L(\R^I_{>0})$ is a finite Boolen combination of  polyhedral cones. Hence the image $\Delta_I(\R^I_{>0})$ is constructible. For each  $\gamma\in \Bbb R^r$ we consider the matrix 
$$B=\left(b_{i,j}(\gamma)\right)\in \mathrm{Mat}( r, |I|;\Q)$$
with $b_{ij}=M_{ij}-\gamma_i N_j$ and the sets
$$B_n:=\left\{\gamma\in \Bbb R^r\mid \mathrm{rank}(B)=n \right\}.$$
Take an element $j_0\in I(\mathfrak f)$ and define a new matrix $B'=\left(b'_{i,j}\right)$ as follows: $b'_{i,j_0}(\gamma):=b_{i,j_0}(\gamma)$, and if $j\neq j_0$ then
$$b'_{i,j}(\gamma):=b_{i,j}(\gamma)-\dfrac{N_j}{N_{j_0}}b_{i,j_0}(\gamma)=M_{ij}-\dfrac{N_j}{N_{j_0}}M_{i,j_0}.$$
Then it is easily seen that
$$B_n=\left\{\gamma\in \Bbb R^r\mid \mathrm{rank}(B')=n \right\}$$
and it is constructible. Let $\{\rho_I\}$ be the constructible partition of $\Q^r_{\geq 0}$ defined as $B_n\cap \Delta_I(\R^I_{>0})$ and $\Q^r_{\geq 0}\setminus\Delta_I(\R^I_{>0})$. We define $\Lambda=\{\rho\}$ to be the intersection of all partitions $\{\rho_I\}$ with $I\subset S, I(\mathfrak f)\neq \emptyset$. We see that for every $\gamma\in B_n\cap \Delta_I(\R^I_{>0})$ one has
$$\Delta_I^{-1}(\gamma)=\Delta_{I,\gamma}\cong  \Bbb R_{>0}^n.$$
Hence the function $\varphi_g$ is constant on each set $\rho$ of the partition $\Lambda$. This gives the constructability of $\varphi_g$.

We denote 
$$V:=\bigsqcup_{\begin{smallmatrix}   I\subseteq S\\  I(\mathfrak{f})\neq \emptyset\end{smallmatrix}} \R^I_{>0}$$
and define a constructible function $$\tilde\varphi\colon V\to \mathscr M_{k}^{\hat{\mu}},$$
as $$\tilde\varphi=\sum_{\emptyset\neq   I\subseteq S}(\L-1)^{|I|-1}\big[\widetilde{E}_I^{\circ}\big] 1_{ \R^I_{>0}}.$$
Then, by Corollary \ref{cor510} and the definition of the Euler integral of $ \tilde\varphi$
\begin{align*}
\MV(X)&=\int_{V} \tilde\varphi\ d\chi_c.
\end{align*}
We also define an algebraic function
$$\Delta\colon V\to  \Bbb R^r_{\geq 0}$$
by $\Delta(\mathbf{k})=\Delta_I(\mathbf{k})$ if $\mathbf{k}:=\left((k_j)_{\j\in I}\right)\in \R^I_{>0}$. Then
$$\Delta_*\tilde\varphi(\gamma)=\varphi_g(\gamma),\ \forall \gamma\in \Q^r_{\geq 0}.$$
Hence
\begin{align*}
\MV(X)&=\int_{V} \tilde\varphi\ d\chi_c=\int_{\Bbb R^r_{\geq 0}}\Delta_* \tilde\varphi\ d\chi_c=\int_{ \Bbb Q^r_{\geq 0}} \varphi_g\ d\chi_c.
\end{align*}
This completes the first equality. To prove the second equality
$$\int_{ \Bbb Q^r_{\geq 0}} \varphi_g d\chi_c=\int_{ \Bbb Q^r_{\geq 0}}\varphi_g d \chi^{\prime},$$
it suffices to show that $\Delta(V)$ is bounded in $\Bbb R^r_{\geq 0}$. But this is obvious since
$$\Delta(V)=\bigsqcup_{\begin{smallmatrix}   I\subseteq S\\  I(\mathfrak{f})\neq \emptyset\end{smallmatrix}} \Delta_I(\R^I_{>0})\subset \{\gamma\in \Bbb R^r_{\geq 0}\mid |\gamma|\leq M\},$$
where
$$M=\max\left\{\dfrac{M_{ij}}{N_j}\mid i\in I, j=1,\ldots,r\right\}.$$
\end{proof}
\begin{corollary}
Let $X$ be a smooth special rigid $K$-varitety with a model $\mathfrak X$. Let \\ $g=\{g_1,\ldots,g_r\}$ be a system of elements of $\Gamma(\mathfrak X,\mathcal{O}_{\mathfrak X})$. For each $\gamma\in  \Q_{\geq 0}$ we define the variety
$$X_{\gamma}:=\left \{x\in X\mid \max_i|g_i(x)| = |\varpi|^{\gamma} \right\}.$$
Then the function $\varphi_{|g|}\colon \Bbb Q_{\geq 0}\to \mathscr M_k^{\hat{\mu}}$ defined as
$$\varphi_{|g|}(\gamma)=\MV(X_{\gamma})$$
is constructible, and moreove, 
$$\MV(X)=\int_{ \Bbb Q^r_{\geq 0}} \varphi_{|g|} d\chi_c.$$
\end{corollary}
\begin{proof}
We first see that the map $|\max|\colon \Q^r_{\geq}\to \Q_{\geq}$ sending $(x_1,\ldots,x_r)$ to $\max_{i}|x_i|$ is a piecewise affine linear map. Let $\varphi_g$ be the map defined as in Theorem \ref{fubini}. Applying Theorem \ref{Qchange} we obtain that the map $\varphi_{|g|}=|\max|_*(\varphi_g)$ is constructible and
$$\MV(X)=\int_{ \Bbb Q^r_{\geq 0}} \varphi_{g} d\chi_c=\int_{ \Bbb Q_{\geq 0}} \varphi_{|g|} d\chi_c.$$
\end{proof}
The following is a direct consequence of the corollary.
\begin{corollary}[Motivic Vanishing Fubini theorem]\label{vanishingfubini}
Let $X$ be a smooth special rigid $K$-varitety with a model $\mathfrak X$. Let $g=\{g_1,\ldots,g_r\}$ be a system of elements of $\Gamma(\mathfrak X,\mathcal{O}_{\mathfrak X})$. For each $\gamma\in  \Q_{\geq 0}$ we define the variety
$$X_{\gamma}:=\left \{x\in X\mid \max_i|g_i(x)| = |\varpi|^{\gamma} \right\}.$$
If $\MV(X_{\gamma})=0$ for all $\gamma\in  \Q_{\geq 0}$, then $\MV(X)=0$.
\end{corollary}
\subsection{Nicaise-Payne's Motivic Fubini theorem for the tropicalization map}\label{NPfubini}
In this section we prove the Nicaise-Payne's Motivic Fubini theorem for the tropicalization map \cite[Theorem 3.1.3]{NP}. The ground field $k$ is assumed to admit all roots of unity. Let us denote by $K'$ the field of Puiseux series $\cup_{n>0} k((t^\frac{1}{n}))$ and by $R'$ its valuation ring.

 Let $X$ be a variety over an algebraic closure $\bar K$ of $K$. A semi-algebraic subset $S$ of $X$ is a finite Boolean combination of subsets of $X(\bar K)$ of the form
\begin{equation}\label{semialgcondition}
\{x \in U(\bar K) \mid \operatorname{val}(f(x)) \leq \operatorname{val}(g(x))\} \subset X(\bar K)
\end{equation}
where $U$ is an affine open subvariety of $X$ and $f, g$ are regular functions on $U$. If $X$ is of the form $X_0 \times_{K} \bar{K}$, for some variety $X_0$ over $K$, then we say that $S$ is {ßem defined over $K$} if we can write it as a finite Boolean combination of sets of the form such that $U, f$, and $g$ in (\ref{semialgcondition}) are defined over $K$. Denoted by $\mathbf{K}\left(\mathrm{VF}_{K}\right) $ a certain Grothendieck ring of semi-alsgebraic sets defined over $K$. Based on the work of Hrushovski and Kazhdan \cite{HK}, Nicaise and Payne have defined the {\em motivic volume}
$$
\mathrm{Vol}: \mathbf{K}\left(\mathrm{VF}_{K}\right) \rightarrow \mathbf{K}^{\widehat{\mu}}\left(\operatorname{Var}_k\right)
$$
 satisfies the following properties.
\begin{itemize}
\item[(1)] Let $X$ be a smooth variety over $K$, and let $\mathcal{X}$ be a smooth $R'$-model of $X':=X\otimes_K K'$ such that the Galois action of $\widehat{\mu}$ on $X'$ extends to a good action on $\mathcal{X}$. Then $S=\mathcal{X}(R')$ is defined over $K$, and $\operatorname{Vol}([S])=\left[\mathcal{X}_k\right]$ in $\mathbf{K}^{\widehat{\mu}}\left(\operatorname{Var}_k\right)$.
\item[(2)] Let $\Gamma$ be a constructible subset of $\mathbb{Q}^n$, for some $n \geq 0$, and set $S^{\prime}=$ $\operatorname{trop}^{-1}(\Gamma)$. Then $S^{\prime}$ is defined over $K$, and
$$
\operatorname{Vol}\left(\left[S^{\prime}\right]\right)=\chi^{\prime}(\Gamma)(\mathbb{L}-1)^n
$$
in $\mathbf{K}^{\widehat{\mu}}\left(\operatorname{Var}_k\right)$.
Here the {\em tropicalization map} $\operatorname{trop}$ is defined as
$$\operatorname{trop}\colon \mathbb{G}^d_{m,K'}\to \Bbb Q^d: (x_1,\ldots,x_d)\mapsto \left(\operatorname{val} (x_1),\ldots,\operatorname{val} (x_d)\right).$$
\end{itemize}
Let $\mathcal X$ be a $R$-scheme of finite type and let $\mathfrak X$ be the completion of $\mathcal X$ along its special fiber $\mathcal X_k$. Let $\mathfrak X_\eta$ be the generic fiber of $\mathfrak X$. Then $\mathfrak X$ is a formal $R$-scheme topologically of finite type and $X$ is a quasi-compact rigid variety. 
\begin{theorem}\label{compareNP}
With the above assumption, the equality
$$\MV(\mathfrak X_\eta)=\operatorname{Vol}([\mathcal X(R')])$$
holds in $\mathscr M_{k}^{\hat{\mu}}$.
\end{theorem}
\begin{proof}
Let $\mathcal Y \to \mathcal X$ be a resolution of singularities of $\mathcal X$ and let $\mathfrak Y$ be the completion of $\mathcal Y$ along its special fiber $\mathcal Y_k$. Then the induced morphism $\mathfrak Y \to\mathfrak X$ ia a resolution of singularities of $\mathfrak X$. Combining \cite[Theorem 2.6.1.]{NP} and Corollary \ref{cor510} we obtain
$$\MV(\mathfrak X_\eta)=\MV(\mathfrak Y_\eta)=\operatorname{Vol}([\mathcal Y(R')])=\operatorname{Vol}([\mathcal X(R')])$$
 in $\mathscr M_{k}^{\hat{\mu}}$.
\end{proof}
Applying Theorem  \ref{fubini} for the system of coordinate functions $x_1,\ldots, x_d$ one obtains the following motivic Fubini theorem for the tropicalization map. Notice that in \cite{NP} the equality (\ref{tropicalization}) is proved to hold in the ring $\mathrm{K}_0^{\hat{\mu}}(\mathrm{Var_k})$. 
\begin{corollary}[Nicaise-Payne's Motivic Fubini theorem for the tropicalization map]\label{tropfubini}
Let $Y$ be a variety over $R$. Let $d$ be a positive integer and let $S$ be a semi-algebraic subset of $\mathbb{G}^d_{m,R}\times_R Y$. Denote by $$\pi\colon \mathbb{G}^d_{m,K}\times_K Y\to \mathbb{G}^d_{m,K}$$
the projection morphism. Then the function
$$(\mathrm{trop}\circ \pi)_* \mathbf{1}_S\colon \mathbb Q^d\to \mathscr M_k^{\hat{\mu}}\colon w\mapsto \mathrm{Vol}\left(S\cap (\mathrm{trop}\circ \pi)^{-1}(w)\right)$$
is constructible, and
\begin{equation}\label{tropicalization}
\mathrm{Vol}(S)=\int_{\Bbb Q^d}(\mathrm{trop}\circ \pi)_* \mathbf{1}_S d\chi'.
\end{equation}
\end{corollary}
\begin{proof}
We prove only for the case when $Y=\Spec K$, the proof of the general case is similar. Assume that $S$ is defined by 
$$\operatorname{val}(f(x)) \leq \operatorname{val}(g(x)),$$
where $f,g\in R[x_1,\ldots,x_d,x_1^{-1},\ldots,x_d^{-1}]$ as in (\ref{semialgcondition}). Let $\mathfrak X$ denote the formal spectrum of the ring 
$$R\{x_1,\ldots,x_d,z\}[x_1^{-1},\ldots,x_d^{-1}]/(zf-g).$$
Then
$$X:=\mathfrak X_\eta=\{x\in E^d\mid 0<|x_i|\leq 1\ \forall i, |f|\geq |g|\}.$$
Consider the system $g=\{x_1,\ldots,x_d\}$ in $\Gamma(\mathfrak X,\mathcal O_\mathfrak X)$ and the function $\varphi_g\colon \Bbb Q^d_{\geq 0}\to \mathscr M_k^{\hat{\mu}}$ defined as in Theorem \ref{fubini}. It follows from Theorem \ref{compareNP} that $\MV(X)=\mathrm{Vol}\left(S\right),$ and for each $w\in \Bbb Q^d_{\geq 0}$
$$\MV(X_w)=\mathrm{Vol}\left(S\cap (\mathrm{trop}\circ \pi)^{-1}(w)\right),$$
where 
$$X_{w}:=\left \{x\in X\mid |x_i(x)| = |\varpi|^{w_i} \right\}.$$
Hence the corollary follows from Theorem \ref{fubini}.

\end{proof}


\subsection{Generalized Poincare series and Rationality}\label{rigpoincare}

Consider a smooth special rigid $K$-variety denoted as $X$, equipped with a model $\mathfrak X$. Let $g=\{g_1,\ldots,g_r\}$ be a system of elements of $\Gamma(\mathfrak X,\mathcal{O}_{\mathfrak X})$. For each $\gamma\in  \Q^r_{\geq 0}$ we define the varieties
$$X_{\gamma}:=\left \{x\in X\mid \max_i|g_i(x)| = |\varpi|^{\gamma} \right\}$$
and
$$X_{\geq\gamma}:=\left \{x\in X\mid |g_i(x)| \leq |\varpi|^{\gamma} \right\}.$$
Assume that $X$ admits an $\mathfrak X$-bounded gauge form $\omega$. Introduce  $\ell(n,m)$ as a linear form on $\Bbb R^2$ with the condition $\ell(n,m)\geq 0$ for $m\leq \gamma n$. We proceed to define the following generalized Poincaré series.
$$P(T):=P(X,\omega,\gamma,\ell;T):=\sum_{0\leq m\leq \gamma n}\L^{-\ell(n,m)}\int_{X_{m/n}(n)}|\omega(n)|\ T^n\text{ in } \mathscr M_{k}^{\hat{\mu}}[\![T]\!]$$

\begin{theorem}
The series  $P(X,\omega,\gamma,\ell;T)$
is rational, its limit is independent of $\ell$ and therefore equal to $-\L^{-d}\MV(X_{\geq \gamma})$.
\end{theorem}

\begin{proof}
Let us consider a resolution of singularities $h\colon \mathfrak Y\to \mathfrak X$ of the formal  $R$-scheme $\mathfrak X$ as in the proof of Theorem \ref{fubini}. Let $\mathfrak E_i$, $i\in S$, be the irreducible components of $(\mathfrak Y_s)_{\mathrm{red}}$.  Let $E_i:={(\mathfrak E_i)}_0$, $ E_I^{\circ}:=\cap_{i\in I}E_i\setminus \cup_{j\not\in I}E_j$ and let  $\widetilde{E}_I^{\circ}\to E_I^{\circ}$ be the covering with Galois group $\mu_{N_I}$ defined locally over $\mathfrak U_0\cap E_I^{\circ}$ as in Definition \ref{resolutionofsingularities}. Then for each $y\in {E}_I^{\circ}$ there exists an affine neighbourhood $\mathfrak U_I$ such that the following identity holds in $\Gamma(\mathfrak U_I,\mathcal{O}_{\mathfrak U_I})$
\begin{align*}
\tilde{\mathfrak{f}}:=h^{*}\mathfrak{f}&=u\prod_{i\in I} y^{N_i}_i\\
\tilde{g_l}:=h^{*}{g_l}&=u_l\prod_{i\in I} y^{M_{il}}_i,\ \forall\ l=1,\ldots,r,
\end{align*}
 where $\mathfrak{f}$ denotes the structural morphism of $\mathfrak X$, $y_i$ is a local equation of $E_i$ at $y$ and $u, u_l$ is invertible in $\Gamma(\mathfrak U_I,\mathcal{O}_{\mathfrak U_I})$. The following lemma is proved by using the same argument as in the proof of Lemma \ref{Poicarecoeff}.
\begin{lemma}\label{Poicarecoeff1}
Let $\omega$ be a $\mathfrak X$-bounded gauge form on $\mathfrak X_\eta$, and put $\alpha_i = \mathrm{ord}_{\mathfrak E_i}\omega$ for each $i\in I$. Then the following identities hold in $\mathscr M_{k}^{\hat{\mu}}$:
\begin{align*}
\int_{X_{\geq\gamma}(n)}|\omega(n)|&=\L^{-d}\sum_{\emptyset\neq   I\subseteq S}(\L-1)^{|I|-1}\big[\widetilde{E}_I^{\circ}\big]\left(\sum_{\begin{smallmatrix} k_i\geq 1, i\in I\\ \sum_{i\in I}k_iN_i=n\\ \min_l\sum_{i\in I}k_i M_{il}\leq \gamma n\end{smallmatrix}}\L^{-\sum_{i\in I}k_i\alpha_i}\right).
\end{align*}
\end{lemma}
Applying the lemma  we have
\begin{align*}
P(T)&=\sum_{m\leq \gamma n}\L^{-\ell(n,m)}\int_{X_{m/n}(n)}|\omega(n)|\ T^n\\
&=\L^{-d}\sum_{m\leq \gamma n}\L^{-\ell(n,m)}\sum_{\emptyset\neq   I\subseteq S}(\L-1)^{|I|-1}\big[\widetilde{E}_I^{\circ}\big] \left(\sum_{\begin{smallmatrix} k_i\geq 1, i\in I\\ \sum_{i\in I}k_iN_i=n\\ \min_l\sum_{i\in I}k_i M_{il}= m\end{smallmatrix}}\L^{-\sum_{i\in I}k_i\alpha_i}\right) T^n\\
&=\L^{-d}\sum_{j=1}^r\sum_{\emptyset\neq   I\subseteq S}(\L-1)^{|I|-1}\big[\widetilde{E}_I^{\circ}\big] S_{I,l}(T),
\end{align*}
where 
$$S_{I,l}(T)=  \sum_{\begin{smallmatrix} k\in \Delta_{I,l}\end{smallmatrix}} \prod_{i\in I}\left(\L^{-\alpha'_{i,l}}T^{N_i}\right)^{k_i}\text{ with } \alpha'_{i,l}=\alpha_i+\ell(N_i,M_{il}),$$
and $ \Delta_{I,l}$ is defined inductively as follows
\begin{align*}
\Delta_{I}&=\{k=(k_1,\ldots,k_{I})\in \Bbb N^I_{>0}\mid \min_{l}\sum_{i\in I}M_{il}k_i\leq \gamma \sum_{i\in I}N_ik_i\},\\
\Delta_{I,1}&=\{k\in \Delta_{I} \mid \sum_{i\in I}M_{i1}k_i\leq \sum_{i\in I}M_{ij}k_i,\ \forall j\geq 2 \},\\
\Delta_{I,l}&=\{k\in \Delta_{I} \mid \sum_{i\in I}M_{i,l}k_i\leq \sum_{i\in I}M_{i,j}k_i,\ \forall j\geq l, \sum_{i\in I}M_{i,l}k_i< \sum_{i\in I}M_{i,j}k_i,\ \forall j<l  \}.
\end{align*}
It follows from (\ref{eq3.1}) that $\lim_{T\to \infty}S_{I,j}(T)=\chi(\Delta_{I,j})$ and therefore
\begin{align*}
\lim_{T\to \infty}P(X,\omega,\gamma,\ell;T)&=\L^{-d}\sum_{j=1}^r\sum_{\emptyset\neq   I\subseteq S}(\L-1)^{|I|-1}\big[\widetilde{E}_I^{\circ}\big]\chi(\Delta_{I,j})\\
&=\L^{-d}\sum_{\emptyset\neq   I\subseteq S}(\L-1)^{|I|-1}\big[\widetilde{E}_I^{\circ}\big]\chi(\Delta_{I}),
\end{align*}
which is clearly independent of $\ell$. On the other hand, as computed above for $\ell=0$, we have
\begin{align*}
P(X,\omega,\gamma,0;T)&=\L^{-d}\sum_{m\leq \gamma n}\sum_{\emptyset\neq   I\subseteq S}(\L-1)^{|I|-1}\big[\widetilde{E}_I^{\circ}\big] \left(\sum_{\begin{smallmatrix} k_i\geq 1, i\in I\\ \sum_{i\in I}k_iN_i=n\\ \min_j\sum_{i\in I}k_i M_{ij}= m\end{smallmatrix}}\L^{-\sum_{i\in I}k_i\alpha_i}\right) T^n\\
&=\L^{-d}\sum_{n\geq 1}\sum_{\emptyset\neq   I\subseteq S}(\L-1)^{|I|-1}\big[\widetilde{E}_I^{\circ}\big] \left(\sum_{\begin{smallmatrix} k_i\geq 1, i\in I\\ \sum_{i\in I}k_iN_i=n\\ \min_j\sum_{i\in I}k_i M_{ij} \leq \gamma n\end{smallmatrix}}\L^{-\sum_{i\in I}k_i\alpha_i}\right) T^n\\
&=P(X_{\geq \gamma},\omega;T).
\end{align*}
Hence
$$\lim_ {T\to\infty}P(X,\omega,\gamma,\ell;T)=\lim_ {T\to\infty}P(X_{\geq \gamma},\omega;T)=-\L^{-d}\MV(X_{\geq \gamma}).$$
\end{proof}


\section{Kontsevich-Soibelman's Integral Identity Conjecute}\label{proofofiic}

In this section, using the theory of motivic integration on rigid varieties developed in the previous section, we prove the following theorem which is known as Kontsevich-Soibelman's Integral Identity Conjecture (\cite{KS}).
\begin{theorem}\label{iic}
Let $f\in k[\![x,y,z]\!]$ with $x=(x_1,\ldots,x_{d_1}), y=(y_1,\ldots,y_{d_2})$ and $z=(z_1,\ldots,z_{d_3})$ be a formal power series such that $f(tx,y,z)=f(x,ty,z)$ in $k[\![x,y,z,t]\!]$. Then $f$ is a series in $k\{x\}[\![y,z]\!]$ and the identity
\begin{align}
\int_{\Bbb A^{d_1}_k} \mathscr S_f=\L^{d_1}  \mathscr S_{\tilde f,0}
\end{align}
holds in $ \mathscr M_k^{\hat{\mu}}$, where $\tilde f(z)=f(0,0,z)\in k[\![z]\!]$.
\end{theorem}
\begin{proof}
Let $\mathfrak X(f)$ and $\mathfrak X(\tilde f)$ be the formal $R$-schemes associated to $f$ and $\tilde f$ respectively. Then  $\mathfrak X(f)_0\cong \Bbb A^{d_1}_k$ and $\mathfrak X(\tilde f)_0=\Spec k$. By (\ref{modeloff})
$$\mathfrak X(f)\cong \Spf\left( R\{x\}[\![y,z]\!]/(f-\varpi)\right) \text{ and }\mathfrak X(\tilde f)\cong \Spf\left( R[\![z]\!]/(\tilde f-\varpi)\right).$$
It follows from Definition \ref{neasrbycycle} that 
$$\mathscr S_f=\MV(\mathfrak X(f))\in \mathscr M_{\Bbb A^{d_1}_k}^{\hat{\mu}}\text{ and } \mathscr S_{\tilde f,0}=\MV(\mathfrak X(\tilde f))\in \mathscr M_k^{\hat{\mu}}.$$
Hence,  the conjecture is equivalent to
\begin{align}\label{identity2}
\MV(\mathfrak X(f)_\eta)=\L^{d_1}  \MV(\mathfrak X(\tilde f)_\eta)\ \text{ in } \mathscr M_k^{\hat{\mu}}.
\end{align}
Note that $\mathfrak X(f)_\eta$ is a closed analytic subvariety of $\mathbf{B}^{d_1}\times D^{d_2+d_3}=\left(\Spf R\{x\}[\![y,z]\!]\right)_\eta$:
$$\mathfrak X(f)_\eta=\left\{u=(x,y,z)\in \mathbf{B}^{d_1}\times D^{d_2+d_3} \mid f(u)=\varpi\right\}.$$
\begin{lemma}\label{lemma41}
Let $X$ be the rigid $K$-variety defined by
$$X:=\left\{u=(x,y,z)\in \mathbf{B}^{d_1}\times D^{d_2+d_3}\mid |f(u)-\varpi|<|\varpi|\right\}.$$
Then $\MV(X)=\MV(\mathfrak X(f)_\eta)$  in $ \mathscr M_k^{\hat{\mu}}$.
\end{lemma}
\begin{proof}
Observe that $X$ is an affine special smooth rigid varieties with a coordinate ring
$$A=R\{x\}[\![y,z,s]\!]/\left(f(x,y,z)-\varpi-\varpi s\right)=R\{x\}[\![y,z,s]\!]/\left(g(x,y,z,s)-\varpi\right)$$
where $g(x,y,z,s)=f(x,y,z)/(1+s)$. Since $1+s$ admits at least one $n$-root for all $n\geq 1$, it follows from Proposition \ref{contactinvariant} that
$$\mathscr S_{\bar{f}}=\mathscr S_g \text{ in }\mathscr M^{\hat{\mu}}_{\mathbb{A}^{d_1}_k}$$
and hence   $\MV(X)=\MV(\mathfrak X(\bar{f})_\eta)$  in $ \mathscr M_k^{\hat{\mu}}$, where $\bar{f}=f$ in $R\{x\}[\![y,z,s]\!]$. Moreover, there is an isormorphism of rigid varieties
$$\mathfrak X(\bar{f})_\eta\cong \mathfrak X(f)_\eta\times D^1,$$
which is induced by the natural isomorphism
$$R\{x\}[\![y,z,s]\!]/\left(\bar{f}(x,y,z,s)-\varpi\right)\cong R\{x\}[\![y,z]\!]/\left(f(x,y,z)-\varpi\right)\otimes_R R[\![s]\!].$$
It follows that
$$\MV(X)=\MV(\mathfrak X(\bar{f})_\eta) =\MV(\mathfrak X({f})_\eta)\cdot \MV(D^1)= \MV(\mathfrak X({f})_\eta)$$
 in $ \mathscr M_k^{\hat{\mu}}$.
\end{proof}
We now decompose $X$ into special rational subdomains $X_0,X'$
$$X_0:=\left\{u\in X \mid |x(u)| \cdot |y(u)|< |\varpi|\right\}$$
and
$$X':=\left\{u\in X \mid |x(u)| \cdot |y(u)|\geq |\varpi|\right\}.$$
Here $|x(u)|:=\max\left\{|x_i(u)| \mid i=1,\ldots,d_1\right\}$ and similarly for $|y(u)|$. 
\begin{lemma}\label{lemma42}
The identity 
$$\MV(X_0)=\L^{d_1}\MV(\mathfrak X(\tilde f)_\eta)$$ holds in $ \mathscr M_k^{\hat{\mu}}$.
\end{lemma}
\begin{proof}
We write $$f(x,y,z)=\sum_{|\alpha|=|\beta|>0}a_{\alpha,\beta,\gamma}x^{\alpha}y^{\beta}z^{\gamma}+\tilde{f}(z).$$
Then, by denoting  $u:=(x,y,z)$,
\begin{align*}
X_0&=\left\{u\in \mathbf{B}^{d_1}\times D^{d_2+d_3}  \mid |f(u)-\varpi|<|\varpi|, |x| \cdot |y|< |\varpi|\right\}\\
&=\left\{u\in  \mathbf{B}^{d_1}\times D^{d_2+d_3}  \mid |\tilde{f}(z)-\varpi|<|\varpi|, |x| \cdot |y|< |\varpi|\right\}\\
&\cong\left\{(x,y)\in  \mathbf{B}^{d_1}\times D^{d_2}  \mid |x| \cdot |y|< |\varpi|\right\}\times\left\{z\in D^{d_3}  \mid |\tilde{f}(z)-\varpi|<|\varpi|\right\}.
\end{align*}
Applying Lemma \ref{lemma42} to $\tilde{f}$, we obtain that
$$\MV\left(\left\{z\in D^{d_3}  \mid |\tilde{f}(z)-\varpi|<|\varpi|\right\}\right)=\MV(\mathfrak X(\tilde f)_\eta).$$
It suffices to prove that $\MV(V)=\L^{d_1}$ with 
$$V:=\left\{(x,y)\in  \mathbf{B}^{d_1}\times D^{d_2}  \mid |x| \cdot |y|< |\varpi|\right\}.$$
In fact, $V=V_1\sqcup V_2$ where 
$$V_1=\left\{(x,y)\in V  \mid |x| \leq |\varpi|\right\}\text{ and } V_2=\left\{(x,y)\in V  \mid |x| > |\varpi|\right\}.$$
We see that 
$$V_1\cong \mathbf{B}^{d_1}(0,|\varpi|)\times D^{d_2}$$
so $\MV(V_1)=\L^{d_1}$. To compute $\MV(V_2)$ we decompose $V_2$ into special rational subdomains 
$$V_2=\bigsqcup^{d_1}_{i=1} W_i,$$
with $W_1:=\{u\in V_2\mid  |x|= |x_1|\}$ and for $1<i\leq d_1$,
$$W_i:=\{(x,y)\in V_2\mid  |x(u)|= |x_i|>|x_j|,\ \forall j<i\}.$$
It is seen that the morphism $W_1 \to W'_1\times \{\xi\in \mathbf{B}^1\mid |\xi|>|\varpi|\}$ which sends $(x,y)$ to $\left((x^{-1}_1x,x_1y),x_1\right)$ is an isomorphism, where
$$W'_1:=\{u\in V_2\mid  x_1=1\}.$$
Therefore $$\MV(W_1)=\MV(W'_1)\cdot \MV\left( \{\xi\in \mathbf{B}^1\mid |\xi|>|\varpi|\}\right)=0.$$
Using the same argument we may prove that $\MV(W_i)=0$ for all $1\leq  i\leq d_1$, and hence 
$$\MV(V_2)=\sum^{d_1}_{i=1} \MV(W_i)=0.$$ This completes the lemma. 
\end{proof}

\begin{lemma}\label{lemma43}
The identity 
$$\MV(X')=0$$ holds  in $ \mathscr M_k^{\hat{\mu}}$.
\end{lemma}
\begin{proof}
We decompode $X$ into special rational subdomains $X=X^{(1)}\sqcup \ldots\sqcup X^{(d_1)}$ where
$X^{(1)}:=\{u\in X\mid  |x(u)|= |x_1(u)|\}$ and, for $1<i\leq d_1$,
$$X^{(i)}:=\{u\in X\mid  |x(u)|= |x_i(u)|>|x_j(u)|,\ \forall j<i\}.$$
Let us consider the system of elements $\{g=(x_iy_j), i=1,\ldots,{d_1},\ j=1,\ldots,{d_2}\}$ in $A=\Gamma(\mathfrak X,\mathcal O_{\mathfrak X})$. 
$$X'^{(i)}:=X^{(i)}\cap X'=\left \{u\in X^{(i)}\mid |g(u)| \geq |\varpi| \right\}.$$
Then, for every $\gamma\in \Q$, the set 
$$X'^{(i)}_\gamma:=\left \{u\in X'^{(i)}\mid |g(u)| = |\varpi|^\gamma \right\}$$
is empty for all $\gamma>1$. Assume that $\gamma\leq 1$, then there is an isomorphism of rigid varieteis
$$X'^{(i)}_\gamma\cong Y^{(i)}_\gamma\times \{\xi\in \mathbf{B}^1\mid |\xi|>|\varpi|^{\gamma}\},$$
sending $(x,y,z)$ to $\left((x^{-1}_ix,x_iy,z),x_i\right)$, where
where
$$ Y^{(i)}_{\gamma}:=\{u\in X'^{(i)}_{\gamma}\mid  x_i(u)=1\}.$$
Since $\MV\left( \{\xi\in \mathbf{B}^1\mid |\xi|>|\varpi|^{\gamma}\}\right)=0$ (see, Example \ref{example31}), it follows that
$\MV(X'^{(i)}_\gamma)=0$ for all $i$. Applying the vanishing Fubini theorem (Corollary \ref{vanishingfubini}) we have 
$\MV(X'^{(i)})=0$ for all $i$, and hence $$\MV(X')=0.$$
The lemma follows.
\end{proof}
Combining Lemmas \ref{lemma41}, \ref{lemma42} and  \ref{lemma43}, we obtain the following identities in $\mathscr M_k^{\hat{\mu}}$
$$\MV(\mathfrak X(f)_\eta)=\MV ( X)=\MV ( X_0)+\MV ( X')=\L^{d_1}\MV(\mathfrak X(\tilde f)_\eta),$$
which give (\ref{identity2}) and hence the theorem.
\end{proof}

\begin{ack}
This article has been written during many visits of the author to Vietnam Institute for Advanced Studies in Mathematics between 2018 and 2024. The authors thank sincerely the institute for their hospitality and valuable supports.
\end{ack}


\end{document}